\documentclass[11pt]{article}

\usepackage{amsmath}
\usepackage{amssymb}
\usepackage{bbm}
\usepackage{amsfonts}
\usepackage{amsthm}
\usepackage{parskip}
\usepackage{tikz}

\newtheorem{thm}{Theorem}[section]
\newtheorem{prop}[thm]{Proposition}
\newtheorem{claim}[thm]{Claim}
\newtheorem{cor}[thm]{Corollary}
\newtheorem*{CCW}{Carbery-Christ-Wright Uniform Sublevel Set Theorem \cite{carbery1999multidimensional}}
\newtheorem*{CCW2}{Proposition (Carbery-Christ-Wright) \cite{carbery1999multidimensional}}
\newtheorem*{stei}{Theorem (Steinerberger) \cite{steinerberger2019sublevel}}

\title{Lower bounds on $L^p$ quasi-norms and the Uniform Sublevel Set Problem}
\author{John Green}
\date{}

\begin{document}
\maketitle
\begin{abstract}
Recently, Steinerberger \cite{steinerberger2019sublevel} proved a uniform inequality for the Laplacian serving as a counterpoint to the standard uniform sublevel set inequality which is known to fail for the Laplacian. In this paper, we observe that many inequalities of this type follow from a uniform lower bound on the $L^1$ norm, and give an analogous result for any linear differential operator, which can fail for non-linear operators. We consider lower bounds on the $L^p$ quasi-norms for $p<1$ as a stronger property that remains weaker than a uniform sublevel set inequality and prove this for the Laplacian and heat operators. We conclude with some naturally arising questions.
\end{abstract}
\begin{footnotesize}
\end{footnotesize}
\section{Introduction}
\subsection{Background}
A central problem throughout analysis is to understand how oscillatory integrals
$$I(\lambda)=\int e^{i\lambda u(x)}\,dx$$
decay for large values of a real ``frequency parameter" $\lambda$, where $u$ is a real-valued ``phase" function. In general, considerations such as the domain of integration or other functions multiplying the oscillatory factor inside the integral are important, but for the sake of this discussion we will not go into these specifics. Typically this decay will be expressed as
$$|I(\lambda)|\leq C\lambda^{-\delta}$$
for some $\delta>0$. Here we have in mind the idea that as $\lambda$ increases, the small differences in $u(x)$ from moving in $x$ become large differences in $\lambda u(x)$, which in turn corresponds to rapid oscillation in $e^{i\lambda u(x)}$. Thus in the integral, we expect $I(\lambda)$ to decay for large $\lambda$ provided $u$ does not stay near any particular value, and the more quickly $u$ ``moves around", the greater the cancellation we expect to occur, and hence the greater we can take $\delta$ to be. Thus one of the most natural conditions to impose is that $Du$ be bounded below by some positive constant, for some differential operator $D$.

Crucial to many applications and key to the discussion in this paper is the idea of uniformity of the constant $C$ within a large class of phases. A natural example appears when studying the Fourier transform of some density on a hypersurface $S$ in $\mathbb{R}^n$. After performing a change of variables, we will have integrals containing an oscillatory factor $e^{-i(x,\phi(x))\cdot\xi}$, where $\phi$ is a function on a piece of $\mathbb{R}^{n-1}$ parametrising a piece of $S$, and $\xi$ is the Fourier variable. Writing $\xi=|\xi|\omega$ for $\omega\in S^{n-1}$, we consider $|\xi|$ to be our frequency parameter and $-(x,\phi(x))\cdot\omega$ is a class of phase functions indexed by $\omega$. If we are to obtain estimates on the Fourier transform of the form $C|\xi|^{-\delta}$, we need to make sure the constant $C$ does not blow-up as we vary over $\omega$. In line with the intuition expressed above, the decay of this Fourier transform is well-known to relate to the curvature of the surface, see Stein \cite{stein1993harmonic} for a discussion of the fundamental results on oscillatory integrals and their relation to the Fourier transform of surface measures.

A related problem is the sublevel set problem: given a real-valued function $u$ and a constant $c$ what conditions should we impose so that estimates of the form $|\{x\in\Omega:|u(x)-c|\leq\varepsilon\}|\leq C\varepsilon^\delta$ hold for appropriate $\Omega$? It is typical to seek estimates independent of $c$ so that the problem is invariant under shifting $u$ by a constant, and we can assume without loss of generality that $c=0$.

That this should be related is apparent from the intuition expressed above, that oscillatory integrals should observe greater cancellation if $u$ does not spend too much time near a given value. And just as in the oscillatory integral case, we are often not only interested in the best possible $\delta$, but also in the uniformity of the constant $C$ in a class of functions. As mentioned above, typically the class of functions for which we seek uniform bounds is those having $Du$ bounded below by some positive constant, where $D$ is a differential operator. For differential operators where $u$ itself does not appear explicitly, in particular for linear differential operators, this condition is invariant under translation of $u$ by a constant, so uniform estimates are necessarily independent of $c$.

We now recall some discussion from the paper of Carbery-Christ-Wright \cite{carbery1999multidimensional}. Oscillatory integral estimates of the form above are known to imply the corresponsing sublevel set estimates. In the case of monomial derivatives, that is, the differential operators $D^\beta=\partial^{\beta_1}_{x_1}\dots\partial^{\beta_n}_{x_n}$ for $\beta=(\beta_1,\dots,\beta_n)\in\mathbb{N}_0^n$, it is known that we can take $\delta=\|\beta\|_1^{-1}=(|\beta_1|+\dots+|\beta_n|)^{-1}$ in the sublevel set problem when $D^\beta u$ is bounded below by a positive constant, and that this is the optimal $\delta$. The same $\delta$ works in the oscillatory integral problem, provided some slightly more restrictive conditions are also imposed, we shall not discuss these here, but refer to Stein \cite{stein1993harmonic}.

However, in all but dimension $1$ (the van der Corput lemma), this $\delta$ has not been shown to hold with a uniform constant without imposing additional assumptions. More precisely, we do not have sublevel set estimates of $C\varepsilon^\delta$ with the optimal $\delta=\|\beta\|_1^{-1}$ and a constant $C$ independent of  the function $u$ satisfying $D^\beta u\geq 1$. The main results of the Carbery-Christ-Wright paper are the following, and again there is an analogue for oscillatory integrals with the slightly more restrictive conditions imposed.

\vspace{0.3cm}

\begin{CCW}
Denote the unit cube in $\mathbb{R}^n$ by $Q_n=[0,1]^n$ and fix $\beta\in\mathbb{N}_0^n$. There exists $C, \delta>0$, depending on $\beta$, such that for any smooth $u$ in a neighbourhood of $Q_n$ having $D^\beta u\geq 1$ on $Q_n$, we have the sublevel set estimates
$$|\{x\in Q_n:|u(x)|\leq\varepsilon\}|\leq C\varepsilon^\delta.$$
Note that $C$ and $\delta$ do not depend on $u$.
\end{CCW}
They also observe that their arguments make sense when $Q_n$ is replaced with different convex sets. Note that this result says nothing about the optimality of $\delta$. It remains open in higher dimensions as to what is the best $\delta$ for which such bounds hold with a uniform constant.

It is worth noting the connections made between the structure of the sublevel sets with certain combinatorial problems, in this direction see also the papers of Katz \cite{katz1999self} and Katz-Krop-Maggioni \cite{katz2002remarks}.

In higher dimensions, we have access to many interesting differential operators, one natural example being the Laplacian. For the Laplacian, one can obtain estimates with $\delta=1/2$ and a non-uniform constant depending on the derivatives of the function up to third order. To see this, one notes that at each point some second order repeated monomial derivative is ``large", and then the mean value theorem can be used to divide the domain into a number of regions depending on the third order derivatives on each of which a repeated monomial derivative remains large, then simply apply the one-dimensional theory - the idea is similar to that used in the proof for monomial derivative estimates given in Stein \cite{stein1993harmonic}. However, in the paper of Carbery-Christ-Wright, it is shown that no uniform estimate can hold for any positive $\delta$.

\vspace{0.3cm}

\begin{CCW2}
For each $\varepsilon\in(0,1/2)$ there exists a smooth $u$ with $\Delta u \equiv 1$ on $[0,1]^2$ but also satisfying the estimate $|\{x\in[0,1]^2:|u(x)|\leq\varepsilon\}|\geq 1-\varepsilon$.
\end{CCW2}
This result extends to higher dimensions by considering this family of counterexamples, and extending them as functions in higher dimensions by asking that they remain constant in the additional variables. Thus no uniform sublevel set estimate holds for any operator consisting of the Laplacian in $2$ or more of the variables plus additional terms - in particular, we observe failure for the heat and wave operators in $2$ or more spatial dimensions.

We shall also see later how this result can be extended to cover the case of the heat operator in one spatial dimension.

The above result for the Laplacian is complimented by the result of Steinerberger \cite{steinerberger2019sublevel}. It states:

\vspace{0.3cm}

\begin{stei}
There exists a constant $c>0$ depending only on $n$ so that if $u:B\rightarrow\mathbb{R}$ satisfies $\Delta u\geq 1$ in $B$, where $B$ is a unit Euclidean ball in $\mathbb{R}^n$, then
$$\|u\|_{L^{\infty}(B)}\cdot|\{x\in B:|u(x)|\geq c\}|\geq c.$$
\end{stei}
This result is in essence saying that if a sublevel set for some small $\varepsilon$ is large, so that its complement is small, then $\|u\|_{L^{\infty}(B)}$ must be large. Applying this to the above family of examples, we see that $u$ must be very large somewhere on the complement of that sublevel set. The intuition for this can be seen from a basic fact which will be crucial to the main results of this paper - by considering the averages of a function $u$ over balls as a function of the radius, we find that the derivative can be quantified exactly in terms of the Laplacian of $u$, indicating that functions with large Laplacian should have ``large" variations, which can be quantified in a uniform way.

It is worth remarking that just as in the Carbery-Christ-Wright Theorem, the assumptions and hence the conclusions of the statement are invariant under replacing $u$ by $u-c$ for any real number $c$.

Steinerberger posed the basic question of whether we can replace $\|u\|_{L^{\infty}(B)}$ with some power of an $L^p$ norm. An affirmative answer to this question is given in the following result.

\vspace{0.3cm}

\begin{prop}\label{general}
Given an open, bounded $\Omega\subseteq\mathbb{R}^n$, there exists a constant $c>0$ depending only on $n$ and $\Omega$ so that if $u:\Omega\rightarrow\mathbb{R}$ satisfies $\Delta u\geq 1$ on $\Omega$, then
$$\|u\|_{L^p(\Omega)}\cdot|\{x\in \Omega:|u(x)|\geq c\}|^{1/p'}\geq c$$
for each $1\leq p\leq \infty$, and $p'$ the conjugate exponent. Note that $c$ does not depend on $p$.
\end{prop}

This result may look less interesting in the case $p=1$, since then $p'=\infty$ and this is just giving a lower bound on the $L^1$ norm. However, the case for every other $p$ follows from this case, although possibly with a different $c$. To see this, let $U_\varepsilon=\{x\in\Omega: |u(x)|<\varepsilon\}$. Then
\begin{align*}c&\leq \int_\Omega |u(x)|\,dx=\int_{\Omega\cap U_\varepsilon} |u(x)|\,dx+\int_{\Omega\cap U^c_\varepsilon} |u(x)|\,dx\\
&\leq \varepsilon|\Omega|+\|u\|_{L^p(\Omega)}\cdot|\{x\in \Omega:|u(x)|\geq \varepsilon\}|^{1/p'}
\end{align*}
where we used H\"older's inequality to bound the second integral. Taking $\varepsilon$ sufficiently small that $\varepsilon|\Omega|\leq c/2$ gives
$$c/2\leq\|u\|_{L^p(\Omega)}\cdot|\{x\in \Omega:|u(x)|\geq \varepsilon\}|^{1/p'}.$$
Note that replacing $c/2$ and $\varepsilon$ with something smaller yields a true inequality, so the result holds with $c'=\min(c/2,\varepsilon)$.

The question of whether we have a lower bound on the $L^1$ norm more generally arises. In fact, we have:

\vspace{0.3cm}

\begin{prop}
Let $D$ be a linear differential operator with smooth coefficients on an open, bounded set $\Omega\subseteq\mathbb{R}^n$. Then there exists $c$ depending only on $n$ and $\Omega$ so that whenever $u:\Omega\rightarrow\mathbb{R}$ satisfies $Du\geq 1$ on $\Omega$ we have $\|u\|_{L^1(\Omega)}\geq c$.
\end{prop}
\begin{proof}
Let $\phi$ be a non-negative smooth function with compact support in $\Omega$ with $\|\phi\|_{L^1(\Omega)}\neq 0$. Then
$$\int_\Omega\phi(x)\,dx\leq\int_\Omega Du(x)\phi(x)\,dx=\int_\Omega u(x)D^*\phi(x)\,dx\leq \|u\|_{L^1(\Omega)}\|D^*\phi\|_{L^\infty(\Omega)}.$$
Here $D^*$ denotes the adjoint operator obtained via integration by parts. We have the desired conclusion with $c=\|\phi\|_{L^1(\Omega)}/\|D^*\phi\|_{L^\infty(\Omega)}$.
\end{proof}

Thus immediately from this proposition and the preceding discussion, we have that Proposition \ref{general} and its analogue for $\Delta$ replaced by any linear differential operator with smooth coefficients holds trivially. It appears then that these results do not use much of the structure of the Laplacian and do not serve as a particularly effective counterpoint to the failure of uniform sublevel set estimates.

One could continue to ask these questions in the non-linear case - we shall make some comments on this later - however, it seems more appropriate to seek a stronger property, one that follows from uniform sublevel set estimates but does not hold in the great generality of the above result. Taking lower bounds on the $L^1$ norm as a motivating property, we consider lower bounds in other $L^p$ norms.

Of course, for $0<p<q<\infty$ we have by H\"older's inequality with exponent $q/p$ that
$$\left(\int_\Omega|u(x)|^p\,dx\right)^{1/p}\leq\left(\int_\Omega|u(x)|^q\,dx\right)^{(p/q)(1/p)}\left(\int_\Omega 1^{q/(q-p)}\,dx\right)^{((q-p)/q)(1/p)}$$
from which we see that $\|u\|_{L^p(\Omega)}\leq\|u\|_{L^q(\Omega)}|\Omega|^{(1/p)-(1/q)}$, and likewise for $q=\infty$. So lower bounds on the $L^p$ quasi-norms for smaller $p$ pose a stronger result and having already established that lower bounds in $L^1$ hold in great generality, we should be particularly interested in the spaces $L^p$ for $0<p<1$. This leads us to the central question of this paper:

\textbf{Question.} Given $p\in(0,1)$ and a (linear) differential operator $D$ on an open, bounded set $\Omega$, does there exist a constant $c_p$ depending only on $n, \Omega$ and $p$ so that whenever $Du\geq 1$ on $\Omega$, we have $\|u\|_{L^p(\Omega)}\geq c_p$?

\subsection{Main results}
Having formulated our main question, we shall give some affirmative answers. Indeed, for the Laplacian, we have:

\vspace{0.3cm}

\begin{thm}\label{LapThm}
Let $\Omega\subseteq\mathbb{R}^{n}$ be open and bounded. For each $p\in(0,1)$, there exists a constant $c_p$ depending only on $n$, $\Omega$ and $p$ such that whenever a function $u:\Omega\rightarrow\mathbb{R}$ satisfies $\Delta u\geq 1$ on $\Omega$ we have the estimate
$$\|u\|_{L^p(\Omega)}\geq c_p.$$
\end{thm}

Moreover, our proof is not a simple existence proof, but rather it illustrates how one can obtain such a constant - although we can say nothing of optimality. The expression for the constant is rather complicated and expressed in terms of a few parameters that arise in the proof, which will be summarised in the final steps of the proofs in section \ref{complete} for reference.

Our proof uses much structure special to functions having $\Delta u \geq 1$. Firstly, we use the derivative formula for averages over balls as mentioned earlier, but we also need a generalisation of the  mean value inequality for subharmonic functions.

Roughly speaking, a function is subharmonic on $\Omega$ if its value at each point $x$ is bounded above by its averages over all balls centred at $x$. By the derivative formula, it will easily be seen that for $C^2$ functions, this is exactly when $\Delta u \geq 0$ throughout $\Omega$. It is known that for locally bounded, non-negative subharmonic functions, we obtain a mean value inequality for $f^p$, $0<p<1$, up to a constant.

These properties have suitable generalisations to the heat operator, and as a result the proof generalises, although we note that in the case of the averages considered, we require some slightly non-standard modifications to get around some technical issues with the usual family of averages for the heat operator.

Let us denote points in $\mathbb{R}^{n+1}$ by $(x,t)\in\mathbb{R}^n\times\mathbb{R}$ and use $\Delta_x$ for the Laplacian in the first $n$ components. The heat operator is $H=\Delta_x-\partial_t$. We have the following:

\vspace{0.3cm}

\begin{thm}\label{HeatThm}
Let $\Omega\subseteq\mathbb{R}^{n+1}$ be open and bounded. For each $p\in(0,1)$, there exists a constant $c_p$ depending only on $n$, $\Omega$ and $p$ such that whenever a function $u:\Omega\rightarrow\mathbb{R}$ satisfies $Hu\geq 1$ on $\Omega$ we have the estimate
$$\|u\|_{L^p(\Omega)}\geq c_p.$$
\end{thm}

As we expect, lower bounds on the $L^p$ quasi-norms follow from uniform sublevel set estimates, since we know that each fixed sublevel set is uniformly small, and hence its complement is uniformly large, which also gives a uniform largeness of the $L^p$ quasi-norm. Indeed, by Chebyshev's inequality, we have for each $0<p<\infty$ and $\varepsilon>0$
$$\varepsilon|\{x\in\Omega:|u(x)|\geq\varepsilon\}|^{1/p}\leq\|u\|_{L^p(\Omega)}.$$
Hence if $u$ is in the class of functions so that $|\{x\in\Omega:|u(x)|\leq\varepsilon\}|\leq C\varepsilon^\delta$, we have $|\{x\in\Omega:|u(x)|\geq\varepsilon\}|\geq |\Omega|-C\varepsilon^\delta\geq |\Omega|/2$ for some small choice of $\varepsilon$, and then we have the desired inequality with $\varepsilon(|\Omega|/2)^{1/p}$.

This creates a strict hierarchy of problems. We have that uniform oscillatory integral estimates imply uniform sublevel set estimates, and these imply lower bounds on each $L^p$ norm, with the $L^1$ case being weak enough that it includes the case of all linear differential operators with sufficiently regular coefficients. It is not yet clear how the case of $L^p$ for $0<p<1$ fits in - does a uniform lower bound on the $L^p$ norm for functions satisfying $Du\geq 1$ hold for any linear differential operator, say with smooth or constant coefficients, or do these results imply something of serious interest regarding the structure of such functions?

In this direction, we shall consider the ``$L^p$ means" for $p\leq 0$, and note some connections to the uniform sublevel set problem which provide evidence for the naturality of our questions.

It is clear, at least, that the proofs in this paper are crucially linked with properties of the Laplacian and heat operators which do not hold in general. It is also worth remarking that we state some intermediate results, some of which are reformulations of existing results to suit this paper, others generalising existing results in a new context, which may suggest further applicability. In particular, we give a straightforward generalisation of a result of Pavlovi\'c \cite{pavlovic1994subharmonic} on passing from mean value inequalities for non-negative functions $f$ to mean value inequalities for $f^p$, $0<p<1$, originally stated for Euclidean balls but which can be extended to deal with different dilation structures, such as the parabolic scaling we will use when considering the heat operator.

The paper will be structured as follows. Section \ref{proofs} contains the proofs of the main results, as well as some other results that are required in the proof and may be of interest more generally. Section \ref{disc} expands on some topics mentioned in this introduction that merit some more in-depth discussion and proofs. Section \ref{ques} collects some questions appearing elsewhere in the paper and poses some further naturally-arising questions.

\section{Proofs of the main results}\label{proofs}
Both of the main theorems will be proven using the same scheme, with the heat operator case requiring some more involved calculations and some more complicated averages, but otherwise nothing inherently more sophisticated. The key ingredient in both cases is a formula for the growth rate of an appropriate family of averages. These formulae are often absorbed into proofs of mean value formulae, so for the sake of completeness we produce them here in a way that emphasises their applicability.

The idea of the proof is as follows. Once we have these formulae for the growth rates in terms of the Laplacian/heat operator, we can say that $u\leq c$ on $\Omega$ implies $u\leq c-d$ on some fixed smaller set, for some $d>0$ depending on the domain but not on $u$ or $c$. If we pick $c>0$ so that $c-d$ is negative, it now follows that either $u>c$ somewhere in $\Omega$ or $u<c-d<0$ everywhere on some fixed smaller set. The latter case trivially gives a lower bound on the $L^p$ norm for any $0<p\leq\infty$.

To deal with the former case, we suppose that there is a point $x$ with $u(x)>c$ that is some fixed distance $R$ from the boundary - just apply the above argument with a different initial set, obtaining a different value of $d$. In turn we choose a different $c$. We then use the basic consequence of the derivative formula - the mean value inequality - which we shall see implies a mean value inequality for $u_+^p$ (where $u_+$ denotes the positive part of $u$). Since $R>0$ is fixed, rearranging the resulting inequality gives a lower bound on the $L^p$ quasi-norm in some subset of $\Omega$, which completes the proof.

In the following proofs $u$ will denote a smooth function\footnote{In fact $C^2$ is enough for this proof, but in view of the remarks of section \ref{exten} on how uniformity allows us to take limits in the inequality, we needn't be too precise.} on a bounded open set $\Omega$.

\subsection{Proof of Theorem \ref{LapThm}, first part}
The derivative formula for the Laplacian is entirely routine and well-known. 

\vspace{0.3cm}

\begin{prop}
Consider for $\overline{B_R(x)}\subseteq\Omega$ the function $\phi:[0,R]\rightarrow\mathbb{R}$ given by $\phi(0)=u(x)$ and the average
$$\frac{1}{|B_r|}\int_{B_r(x)}u(y)\,dy$$
for $0<r\leq R$. Clearly $\phi$ is continuous on $[0,R]$, and on the open interval $(0,R)$ we have
\begin{equation}\label{deriv1}
\phi'(r)=\frac{1}{|B_r|}\int_{B_r(x)}\frac{r^2-|x-y|^2}{2r}\Delta u(y)\,dy.
\end{equation}
\end{prop}
Here $B_r(x)$ denotes the Euclidean ball of radius $r$ centred at $x$, and $|B_r|$ is the Lebesgue measure of the ball.
\begin{proof}
Using the change of variables $y=x-r\tilde{y}$ we have
$$\phi(r)=\frac{1}{|B_1|}\int_{B_1(0)}u(x-r\tilde{y})\,d\tilde{y}.$$
Differentiating under the integral and changing variables again we have 
\begin{align*}
\phi'(r)&=\frac{1}{|B_1|}\int_{ B_1(0)}\nabla u(x-r\tilde{y})\cdot (-\tilde{y})\,d\tilde{y}\\
&=\frac{1}{|B_r|}\int_{ B_r(0)}\nabla u(y)\cdot \frac{y-x}{r}\,dy.
\end{align*}
Now using polar coordinates, with $\sigma_s$ being the induced surface measure on the sphere of radius $s$, we have
\begin{align*}
\phi'(r)&=\frac{1}{|B_r|}\int_0^r\int_{\partial B_r(x)}\nabla u(y)\cdot \frac{y-x}{r}\,d\sigma_s(y)\,ds\\
&=\frac{1}{|B_r|}\int_0^r\frac{s}{r}\int_{\partial B_s(x)}\nabla u(y)\cdot \frac{y-x}{s}\,d\sigma_s(y)\,ds\\
&=\frac{1}{|B_r|}\int_0^r\frac{s}{r}\int_{B_s(x)}\Delta u(y)\,dy\,ds
\end{align*}
where in the last step we used Gauss' Divergence Theorem. We again apply polar coordinates to the inner integral and apply Fubini's Theorem.
\begin{align*}
\phi'(r)&=\frac{1}{|B_r|}\int_0^r\frac{s}{r}\int_0^s\int_{\partial B_t(x)}\Delta u(y)\,d\sigma_t(y)\,dt\,ds\\
&=\frac{1}{|B_r|}\int_0^r\int_t^r\frac{s}{r}\int_{\partial B_t(x)}\Delta u(y)\,d\sigma_t(y)\,ds\,dt\\
&=\frac{1}{|B_r|}\int_0^r\int_{\partial B_t(x)}\frac{r^2-t^2}{2r}\Delta u(y)\,d\sigma_t(y)\,dt.
\end{align*}
Noting $t=|x-y|$ on $\partial B_t(x)$ completes the calculation.
\end{proof}

For later reference, we note the following immediate corollary:

\vspace{0.3cm}

\begin{cor}\label{LapCor}
Whenever $u:\Omega\rightarrow\mathbb{R}$ satisfies $\Delta u\geq 0$ in $\Omega$ and $\overline{B_R(x)}\subseteq\Omega$, we have
$$u(x)\leq\frac{1}{|B_R|}\int_{B_R(x)}u(y)\,dy$$
and the same inequality is also true for $u$ replaced with $u_+(x)=\max(u(x),0)$
\end{cor}
The second statement follows from the first because the average of $u_+$ is no less than the average of $u$, and $0$ is always no greater than the average of the non-negative function $u_+$.

\vspace{0.3cm}

\begin{claim}\label{LapClaim}
Let $u:\Omega\rightarrow\mathbb{R}$ satisfy $\Delta u\geq 1$ in $\Omega$ and let $\Omega_R$ denote the set of points $x$ in $\Omega$ so that $\overline{B_R(x)}\subseteq\Omega$. Then if $u\leq c$ on $\Omega$, we have that $u\leq c-K_nR^2$ on $\Omega_R$, where
$$K_n=1/(2n+4)$$
\end{claim}

\begin{proof}
This follows from the derivative formula (\ref{deriv1}). We shall denote by $|\partial B_s|=|\partial B_1|s^{n-1}$ the total surface measure of the sphere of radius $s$, and using the assumption $\Delta u\geq 1$, we have that the derivative $\phi'(r)$ is bounded below by
\begin{align*}
\frac{1}{|B_r|}\int_0^r|\partial B_1|s^{n-1}\frac{r^2-s^2}{2r}\,ds&=\frac{|\partial B_1|}{|B_r|}\left(\frac{r^{n+2}}{2rn}-\frac{r^{n+2}}{2r(n+2)}\right)\\
&=\frac{|\partial B_1|}{|B_1|}\left(\frac{1}{2n}-\frac{1}{2(n+2)}\right)r\\
&=n\left(\frac{1}{2n}-\frac{1}{2(n+2)}\right)r\\
&=\frac{r}{n+2}.
\end{align*}
Now whenever $x\in\Omega_R$, we have by the fundamental theorem of calculus that $u(x)=\phi(0)=\phi(R)-\int_0^R\phi'(r)\,dr$, which along with $\Delta u\geq 1$ gives
\begin{align*}
u(x)&=\frac{1}{|B_R|}\int_{B_R(x)}u(y)\,dy-\int_0^R\frac{1}{|B_r|}\int_{B_r(x)}\frac{r^2-|x-y|^2}{2r}\Delta u(y)\,dy\,dr\\
&\leq c-\int_0^R\frac{r}{n+2}\,dr\leq c-\frac{R^2}{2n+4}=c-K_nR^2.
\end{align*}
\end{proof}

Corollary \ref{LapCor} paired with Claim \ref{LapClaim} and the discussion of mean value inequalities in section \ref{mvi} are the only ingredients used in the proof of Theorem \ref{LapThm}.

\subsection{Proof of Theorem \ref{HeatThm}, first part}\label{heatstart}
For the heat operator case, we shall first define some notation. Let $\Phi_n$ be the standard heat kernel on $\mathbb{R}^{n+1}$,
$$\Phi_n(x,t):=\frac{1}{(4\pi t)^{n/2}}e^{-|x|^2/4t}.$$
The heatball centred at $(x,t)$ of radius $r>0$ is the compact set
$$E(x,t;r):=\{(x,t)\}\cup\{(y,s)\in\mathbb{R}^{n+1}:s\leq t,\,\Phi_n(t-s,x-y)\geq 1/r^n\}.$$
Note that except for at $(x,t)$, $\Phi_n(t-s,x-y)=1/r^n$ on the boundary.

For convenience we shall also use $E(r)$ to denote $\{(x,t):\Phi_n(x,t)\geq 1/r^n\}$. Notice that $|E(x,t;r)|=|E(r)|=r^{n+2}|E(1)|$, a fact that follows from the parabolic scaling $(y,s)\mapsto (ry,r^2s)$ taking $E(1)$ to $E(r)$.

Whenever the heatball $E(x,t;R)$ is in $\Omega$, we shall consider the functions $\phi:[0,R]\rightarrow \mathbb{R}$ given by $\phi(0)=u(x,t)$ and
$$\phi(r)=\frac{1}{4r^n}\int_{E(x,t;r)}u(y,s)\frac{|x-y|^2}{(t-s)^2}\,dy\,ds$$
for $0<r\leq R$. These averages were considered in a paper of Watson \cite{watson1973theory}, in which a theory of subtemperatures, analogous to the theory of subharmonic functions, was developed.

In the paper of Watson \cite{watson1973theory}, it is seen that
$$\frac{1}{4r^n}\int_{E(x,t;r)}\frac{|x-y|^2}{(t-s)^2}\,dy\,ds=1$$
which is where the normalisation in the above definition comes from. However, as the precise normalisation will not be important for us, and in order to give a self-contained treatment in this paper, the reader may wish instead to think of the factor of $4r^n$ being written as
$$V(r):=\int_{E(r)}\frac{|y|^2}{s^2}\,dy\,ds=r^nV(1).$$
The equality $V(r)=r^nV(1)$ follows from parabolic scaling. That this quantity is finite can be seen quite easily in higher dimensions, and the techniques we will use are suggestive of methods used to account for unboundedness of the kernel of the averages, $|x-y|^2/(V(r)(t-s)^2)$. In fact, we will only need this result in higher dimensions.

By rearranging the formula defining the level set in $E(r)$, we can give the equivalent expression $\{(y,s):0<s\leq r^2/4\pi,|y|\leq(2ns\log(r^2/4\pi s))^{1/2}\}$. The integral becomes
\begin{align*}
V(r)&=\int_0^{r^2/4\pi}\int_{|y|\leq(2ns\log(r^2/4\pi s))^{1/2}}\frac{|y|^2}{s^2}\,dy\,ds\\
&=\int_0^{r^2/4\pi}\int_{|y|\leq 1}\frac{|y|^2}{s^2}(2ns\log(r^2/4\pi s))^{(n+2)/2}\,dy\,ds\\
&=\int_{|y|\leq 1}|y|^2\,dy\,\int_0^{r^2/4\pi}\left(2ns^{1-\frac{4}{n+2}}\log(r^2/4\pi s)\right)^{(n+2)/2}\,ds.
\end{align*}
The $y$ integral is clearly finite and easy to compute. For $n\geq 3$, the integrand in the $s$ integral is bounded, extending continuously to $s=0$ - indeed, it is clear that $s^{1-\frac{4}{n+2}}\log(r^2/4\pi s)$ tends to $0$ as $s\rightarrow 0$.

Now, we have already observed that the kernel for these averages is unbounded, which will cause problems when we try to establish mean value inequalities for $u_+^p$. Nevertheless, we will later consider a modified kernel which is bounded, and hence the associated measure can be controlled by a multiple of the Lebesgue measure, so that the arguments of section \ref{complete} will work. The approach will be to add in some extra spatial variables, apply the derivative formula in higher dimensions, but integrate out the extra variables appearing in the kernel to get a new, bounded kernel in lower dimensions, in a manner not unlike the above calculation. 

First, however, we shall give a proof of the derivative formula for heatballs. This proof is due to Evans \cite{evans10}, but it is not explicitly given there, instead being absorbed into the proof of the corresponding mean value formula for the heat equation. 

\vspace{0.3cm}

\begin{prop}
The function $\phi$ defined above is continuous on $[0,R]$, and in the interval $(0,R)$ its derivative is given by
\begin{equation}\label{deriv2}
\phi'(r)=\frac{n}{r^{n+1}}\int_{E(x,t;r)}Hu(y,s)\log(r^n\Phi_n(x-y,t-s))\,dy\,ds.
\end{equation}
\end{prop}

\begin{proof}
Using the translation and rescaling $y=x-r\tilde{y},s=t-r^2\tilde{s}$, we have
$$\phi(r)=\frac{1}{4}\int_{E(1)}u(x-r\tilde{y},t-r^2\tilde{s})\frac{|\tilde{y}|^2}{\tilde{s}^2}\,d\tilde{y}\,d\tilde{s}.$$
Continuity of $\phi$ is obvious from the smoothness of $u$ in a neighbourhood of $E(x,t;r)$, and noting that the normalisation $V(1)=4$ in these averages is the correct one. We can differentiate under the integral to obtain
$$\phi'(r)=-\frac{1}{4}\int_{E(1)}\left(2rsu_t+\sum_{i=1}^nu_{x_i}y_i\right)\frac{|y|^2}{s^2}\,dy\,ds$$
suppressing the argument $(x-ry,t-r^2s)$ of $u_{x_i}$ and $u_t$. Scaling back we have
$$\phi'(r)=-\frac{1}{4r^{n+1}}\int_{E(r)}\left(2su_t+\sum_{i=1}^nu_{x_i}y_i\right)\frac{|y|^2}{s^2}\,dy\,ds$$
with the argument $(x-y,t-s)$ suppressed. For convenience let us denote $\psi(y,s)=\log(r^n\Phi_n(y,s))=n\log(r)-(n/2)\log(4\pi s)-|y|^2/4s$. We have $\psi_{x_i}(y,s)=-y_i/2s$ and thus
$$\int_{E(r)}2su_t\frac{|y|^2}{s^2}\,dy\,ds=-4\int_{E(r)}u_t\sum_{i=1}^n\psi_{x_i}(y,s)y_i\,dy\,ds.$$
Noting that $\psi(y,s)=0$ on the boundary of $E(r)$, we can use integration by parts $\psi_{x_i}$ and $y_iu(x-y,t-s)$ to get that this equals
$$4\int_{E(r)}\psi\sum_{i=1}^n(u_t-y_iu_{tx_i})\,dy\,ds=4\int_{E(r)}\psi\left(nu_t-\sum_{i=1}^ny_iu_{tx_i}\right)\,dy\,ds.$$
Focusing on the second term, we have
\begin{align*}
&-4\int_{E(r)}\sum_{i=1}^n\psi(y,s)y_i\frac{\partial}{\partial s}(-u_{x_i}(x-y,t-s))\,dy\,ds\\
=\,&-4\int_{E(r)}\sum_{i=1}^n\psi_t(y,s)u_{x_i}(x-y,t-s)y_i\,dy\,ds\\
=\,&-4\int_{E(r)}\sum_{i=1}^n\left(\frac{|y|^2}{4s^2}-\frac{n}{2s}\right)u_{x_i}(x-y,t-s)y_i\,dy\,ds
\end{align*}
and hence
$$\int_{E(r)}2su_t\frac{|y|^2}{s^2}\,dy\,ds=4\int_{E(r)}\left[nu_t\psi-\sum_{i=1}^n\left(\frac{|y|^2}{4s^2}-\frac{n}{2s}\right)u_{x_i}y_i\right]dy\,ds.$$
All together this gives
\begin{align*}
\phi'(r)&=\frac{-1}{4r^{n+1}}\int_{E(r)}\left[4nu_t\psi-4\sum_{i=1}^n\left(\frac{|y|^2}{4s^2}-\frac{n}{2s}\right)u_{x_i}y_i+\sum_{i=1}^nu_{x_i}y_i\frac{|y|^2}{s^2}\right]dy\,ds\\
&=\frac{-1}{4r^{n+1}}\int_{E(r)}\left[4nu_t\psi-4\sum_{i=1}^n\frac{-n}{2s}u_{x_i}y_i\right]dy\,ds\\
&=\frac{n}{r^{n+1}}\int_{E(r)}\left[\left(\sum_{i=1}^n\frac{-y_i}{2s}u_{x_i}(x-y,t-s)\right)-u_t(x-y,t-s)\psi\right]dy\,ds.
\end{align*}
Since $\psi_{x_i}(y,s)=-y_i/2s$ and $\psi$ is $0$ on the boundary of $E(r)$, an integration by parts in $y_i$ for each term of the sum yields
\begin{align*}
\phi'(r)&=\frac{n}{r^{n+1}}\int_{E(r)}\Delta u(x-y,t-s)-u_t(x-y,t-s)\psi\,dy\,ds\\
&=\frac{n}{r^{n+1}}\int_{E(r)}Hu(x-y,t-s)\psi(y,s)\,dy\,ds\\
&=\frac{n}{r^{n+1}}\int_{E(x,t;r)}Hu(y,s)\log(r^n\Phi(x-y,t-s))\,dy\,ds
\end{align*}
as desired.
\end{proof}

As before, we have the following immediate corollary

\vspace{0.3cm}

\begin{cor}\label{HeatCor}
Whenever $u:\Omega\rightarrow\mathbb{R}$ satisfies $Hu\geq 0$ in $\Omega$ and the heatball $E(x,t;R)\subseteq\Omega$, we have
$$u(x,t)\leq\frac{1}{4r^n}\int_{E(x,t;R)}u(y,s)\frac{|x-y|^2}{(t-s)^2}\,dy\,ds$$
and the same inequality is also true for $u$ replaced with $u_+(x)=\max(u(x),0)$.
\end{cor}

\vspace{0.3cm}

\begin{claim}\label{HeatClaim}
Let $u:\Omega\rightarrow\mathbb{R}$ satisfy $Hu\geq 1$ in $\Omega$ and let $\Omega_R$ denote the set of points $(x,t)$ in $\Omega$ so that $E(x,t;R)\subseteq\Omega$. Then if $u\leq c$ on $\Omega$, we have that $u\leq c-K_nR^2$ on $\Omega_R$, where
$$K_n=(n^2/2)|E(0,0;1)|e^{n+2}$$
\end{claim}

\begin{proof}
By the fundamental theorem of calculus for $\phi$ and the derivative formula (\ref{deriv2}), we have
\begin{align*}
u(x,t)&=\phi(0)=\phi(R)-\int_0^R\phi'(r)\,dr\\
&=\frac{1}{4r^n}\int_{E(x,t;r)}u(y,s)\frac{|x-y|^2}{(t-s)^2}\,dy\,ds\\
&-\int_0^R\frac{n}{r^{n+1}}\int_{E(x,t;r)}Hu(y,s)\log(r^n\Phi_n(x-y,t-s))\,dy\,ds\,dr.
\end{align*}
Note that $\log(r^n\Phi_n(x-y,t-s))$ is positive on the heatball $E(x,t;r)$, and is at least $n$ on the heatball $E(x,t;r/e)$. Hence we have the bound
$$\int_{E(x,t;r)}Hu(y,s)\log(r^n\Phi_n(x-y,t-s))\,dy\,ds\geq n|E(x,t;r/e)|.$$
The right hand side is equal to $n|E(x,t;1)|(r/e)^{n+2}=n|E(0,0;1)|(r/e)^{n+2}=:C_nr^{n+2}$. We use this bound together with the bound $\phi(R)\leq c$ (which follows from $u\leq c$ in $\Omega$) to obtain that for $(x,t)\in\Omega_R$,
$$u(x,t)\leq c-\int_0^R C_nnr\,dr=c-K_nR^2.$$
\end{proof}

We stress again that Corollary \ref{HeatCor} and Claim \ref{HeatClaim} are not sufficient as is, which is why we must introduce one more ingredient - averages over the so-called ``modified heatballs". This idea was used by Kuptsov \cite{kuptsov1981mean}, our treatment follows a paper by Watson \cite{watson2002elementary} giving a review of the main ideas and some further results. We shall only require some basic facts, which we summarise here.

Fix $m\geq 3$. Starting with a function $u$ on an open subset $\Omega$ of $\mathbb{R}^{n+1}$, we examine the above averages $\phi$ for a function $\tilde{u}$ on $\mathbb{R}^m\times\Omega$, defined by extending $u$ independent of the first $m$ variables - that is, for $\xi\in\mathbb{R}^m$, $(x,t)\in\Omega$, set $\tilde{u}(\xi,x,t)=u(x,t)$ and consider the averages $\phi$ for $\tilde{u}$, which clearly take the form
$$\phi(r)=\frac{1}{4r^{m+n}}\int_{E(\xi,x,t;r)}\frac{|x-y|^2+|\xi-\eta|^2}{(t-s)^2}u(y,s)\,d\eta\,dy\,ds.$$
Since $u$ does not depend on $\eta$, we can carry out the integration in $\eta$. As above, we rearrange the superlevel set formula defining the heatball to observe that $E(\xi,x,t;r)$ is the set of $(\eta,y,s)$ satisfying
$$0\leq t-s\leq r^2/4\pi,|x-y|^2+|\xi-\eta|^2\leq 2(m+n)(t-s)\log(r^2/4\pi(t-s)).$$
The $(n,m)$-modified heatball is the projection of $E(\xi,x,t;r)$ onto the last $n+1$ coordinates, denoted $E_m(x,t;r)$, and by carrying out the integration in $\eta$ we obtain a new kernel on $E_m(x,t;r)$ that will be bounded for large enough $m$, as we shall demonstrate. Explicitly, we may write 
\begin{align*}
E_m(x,t;r)&=\{(y,s):\Phi_{m+n}(0,x-y,t-s)\geq 1/r^{m+n}\}\\
&=\{(x-ry,t-r^2s):\Phi_{m+n}(0,y,s)\geq 1\}.
\end{align*}
For fixed $y$, $s$, we must integrate over $|\xi-\eta|\leq A=A(x-y,t-s)$, given by
$$A(x-y,t-s):=\left(2(t-s)(m+n)\log\left(\frac{r^2}{4\pi(t-s)}\right)-|x-y|^2\right)^{1/2}.$$
We compute the integral in polar coordinates
\begin{align*}
\int_{|\xi-\eta|\leq A}\frac{|\xi-\eta|^2+|x-y|^2}{(t-s)^2}\,d\eta&=|\partial B_1|\int_0^A\frac{r^2+|x-y|^2}{(t-s)^2}r^{m-1}\,dr\\
&=\frac{|\partial B_1|}{(t-s)^2}\left(\frac{A^{m+2}}{m+2}+|x-y|^2\frac{A^m}{m}\right)\\
&=\frac{|B_1|}{(t-s)^2}A^m\left(\frac{mA^2}{m+2}+|x-y|^2\right)
\end{align*}
where we used the basic formula $|\partial B_1|=m|B_1|$. Hence we may write
\begin{align*}
\phi(r)&=\frac{1}{4r^{m+n}}\int_{E_m(x,t;r)}\frac{|B_1|}{(t-s)^2}A^m\left(\frac{mA^2}{m+2}+|x-y|^2\right)u(y,s)\,dy\,ds\\
&=\frac{1}{r^{n+2}}\int_{E_m(x,t;r)}\frac{|B_1|r^{2-m}}{4(t-s)^2}A^m\left(\frac{mA^2}{m+2}+|x-y|^2\right)u(y,s)\,dy\,ds.
\end{align*}
We will see that the non-negative quantity
$$\frac{|B_1|r^{2-m}}{4(t-s)^2}A^m\left(\frac{mA^2}{m+2}+|x-y|^2\right)$$
is bounded above on $E_m(x,t;r)$ by some constant independent of $x$, $t$ and $r$. Indeed, noting that the points of $E_m(x,t;r)$ are of the form $(x-ry,t-r^2s)$ for $y$ and $s$ satisfying $\Phi_{m+n}(0,y,s)\geq 1$, we may rewrite the expression as
$$\kappa_{m,n}(y,s):=\frac{|B_1|}{4s^2}[\tilde{A}(y,s)]^m\left(\frac{m[\tilde{A}(y,s)]^2}{m+2}+|y|^2\right)$$
where $\Phi_{m+n}(0,y,s)\geq 1$ and
$$\tilde{A}(y,s)=\left(2s(m+n)\log\left(\frac{1}{4\pi s}\right)-|y|^2\right)^{1/2}.$$
Substituting the value of $\tilde{A}^2$ into the brackets gives
$$\kappa_{m,n}(y,s)=\frac{|B_1|}{2m+4}\tilde{A}^m\left(\frac{m(m+n)}{s}\log\left(\frac{1}{4\pi s}\right)+\frac{|y|^2}{s^2}\right).$$
It is clear that this is continuous in $(y,s)$ on the set where $\Phi_{m+n}(0,y,s)\geq 1$, so once we show the limit exists as $(y,s)\rightarrow(0,0)$ in the superlevel set, we will have that $\kappa_{m,n}$ extends continuously to the reflection of the modified heatball $E_m(0,0;1)$, that is, $\{(0,0)\}\cup\{(y,s):\Phi_{m+n}(0,y,s)\geq 1\}$, a compact set, and is therefore bounded.

Observe that $\tilde{A}^m$ is bounded by $(2s(m+n)\log(1/4\pi s))^{m/2}$, and also that in (the reflection of) the modified heatball we have $|y|^2\leq 2s(m+n)\log(1/4\pi s)$, so $\kappa_{m,n}$ can be bounded by a constant depending on $n$ and $m$ times
$$s^{(m-2)/2}\log\left(\frac{1}{4\pi s}\right)^{(m+2)/2}.$$
Similarly to before, we can easily see that as $s$ goes to $0$, provided $m\geq 3$, this quantity goes to $0$, hence $\kappa_{m,n}$ extends continuously to $(0,0)$ and is thus bounded above.

We note that as we are to maximise $\kappa_{m,n}$ over the set
\begin{align*}
E:=&\{(0,0)\}\cup\{(y,s):\Phi_{m+n}(0,y,s)\geq 1\}\\
=&\{(y,s)\in\mathbb{R}^{n+1}:0\leq s\leq 1/4\pi,|y|^2\leq 2s(m+n)\log(1/4\pi s)\},
\end{align*}
and from the latter expression it is apparent that $\kappa_{m,n}$ is equal to $0$ on the boundary of $E$, and differentiable in the interior, we see that the maximum is attained in the interior of $E$, thus may be found by basic calculus. The determination of this maximum, which we shall denote by $M_{m,n}$, is entirely routine and so has been placed into Appendix \ref{maxcalc}. We have
$$M_{m,n}=|B_1|\frac{2\pi}{e}\left(\frac{2(m+n)(m+2)}{(4\pi e)(m-2)}\right)^{m/2}\left(\frac{m(m+n)}{m-2}\right).$$

\vspace{0.3cm}

\begin{prop}\label{HeatMod}
Let $m\geq 3$ and let $M_{m,n}$ be as above. Suppose $u:\Omega\rightarrow\mathbb{R}$ satisfies $Hu\geq 0$ on $\Omega$ and let $u_+$ be the positive part of $u$, that is, $u_+(y,s)=\max(u(y,s),0)$. Then whenever the modified heatball $E_m(x,t;R)$ is contained in $\Omega$, we have
$$u_+(x,t)\leq\frac{M_{m,n}}{R^{n+2}}\int_{E_m(x,t;R)}u_+(y,s)\,dy\,ds$$
\end{prop}

We remark that it is clear that the modified heatball is smaller when $m$ is smaller, thus in this proposition, the conclusion is strongest when $m$ and $M_{m,n}$ are as small as possible. We shall make no attempt to optimise this for any particular $n$.

\begin{proof}
It follows from Corollary \ref{HeatCor} that for $\xi\in\mathbb{R}^m$,
$$u_+(x)\leq\frac{1}{4R^{m+n}}\int_{E(\xi,x,t;R)}u_+(y,s)\frac{|x-y|^2+|\xi-\eta|^2}{(t-s)^2}\,d\eta\,dy\,ds.$$
We have just seen that the right hand side can be rewritten as
$$\frac{1}{R^{n+2}}\int_{E_m(x,t;R)}\frac{|B_1|r^{2-m}}{4(t-s)^2}A^m\left(\frac{mA^2}{m+2}+|x-y|^2\right)u_+(y,s)\,dy\,ds$$
which can be bounded above by
$$\frac{M_{m,n}}{R^{n+2}}\int_{E_m(x,t;R)}u_+(y,s)\,dy\,ds$$
since $u_+$ is non-negative.
\end{proof}

To conclude the proofs of Theorems \ref{LapThm} and \ref{HeatThm}, it remains to give some discussion of mean value inequalities.

\subsection{Mean value inequalities}\label{mvi}
We shall generalise a theorem of Pavlovi\'c \cite{pavlovic1994subharmonic}, which for non-negative functions $f$ will allow us to pass from mean value inequalities for $f$ to mean value inequalities for $f^p$, $0<p<1$. The formulation below is not the most general, but sufficient for our purposes, and requires only minor adjustments to Pavlovi\'c's proof.

The result was originally proven by Pavlovi\'c for averages over Euclidean balls, but we will see that the result remains true for averages associated to a system of objects with associated centres and radii. We obtain this system by fixing some set and group of dilations, and constructing this system by dilating then translating the initial set.

Remarks:
\begin{itemize}
\item One can also easily establish that for non-negative functions $f$ satisfying mean value inequalities, a mean value inequality holds for $f^p$, $1<p<\infty$, by using Jensen's inequality.
\item There are many proofs of this result for Euclidean balls, we have chosen a simple proof that leads easily to some helpful generalisations.
\item There are extensions of this result on manifolds that seem suitable for generalisation, see Li \& Schoen \cite{li1984p}.
\item This proof cannot effectively deal with unbounded weights such as those naturally arising for heatballs, hence the modified heatballs considered in the previous subsection.
\end{itemize}

\textbf{Set-up.} Fix an open, bounded set $B\subseteq\mathbb{R}^n$ containing $0$. We will call this the unit ball at $0$. Consider a group of dilations $D=\{\delta_r:r>0\}$ given by $\delta_r(x_1,\dots,x_n)=(r^{\lambda_1}x_1,\dots,r^{\lambda_n}x_n)$ for some positive numbers $\lambda_1,\dots,\lambda_n$. Define the degree of homogeneity $A=\lambda_1+\dots+\lambda_n$. Denote by $B_r$ the image of $B$ under $\delta_r$, and call this the ball of radius $r$ at $0$. Translation of $B_r$ by a fixed element $a\in\mathbb{R}^n$ will be denoted $B_r(a)$, and we will call this the ball of radius $r$ at $a$. We will call the collection of balls arising in this way a system of balls associated to the pair $(B,D)$.

Note that
$$\delta_s(B_r(a))=\delta_s(a+B_r)=\delta_s(a+\delta_r(B))=\delta_s(a)+\delta_s\delta_r(B)=B_{rs}(\delta_s(a)).$$

\textbf{Initial assumption.} We will ask that on $B$ there exists a ``radius function" with good properties. Namely, we ask for the existence of a bounded, positive measurable function $R$ on $B$ so that $\overline{B_{R(a)}(a)}\subseteq B$, and moreover, whenever $x\in \overline{B_{R(a)}(a)}$, we have $R(a)/R(x)$ bounded above by some constant $K$ independent of $a$ and $x$. In fact, we will show that a radius function always exists in this set-up.

Observing that it is sufficient to suppose the existence of radius function with this property is the key idea that generalises Pavlovi\'c's proof. For Euclidean balls, they used $R(a)=(1-|a|)/2$ for $a\in B$ in place of what we call a radius function. This fits into the framework considered here in the obvious way - the dilation structure is just the usual scaling by $r$, and clearly $R(a)/R(x)$ is bounded above by $2$ on $B_{R(a)}(a)$, since $|x|\leq (1+|a|)/2$ for $x\in B_{R(a)}(a)$, and so $1-|x|\geq(1-|a|)/2$, which rearranges to give the bound.

\textbf{Notation.} For a given $C>0$, consider the set $S(B,D,C)$ of functions satisfying mean value inequalities with constant $C$ for the system of balls associated to $(B,D)$. That is, $S(B,D,C)$ is the set of those locally bounded non-negative functions $f$ defined on open subsets $\Omega$ of $\mathbb{R}^n$ with
$$f(a)\leq \frac{C}{r^A}\int_{B_r(a)}f(x)\,dx$$
whenever $\overline{B_r(a)}\subseteq \Omega$. Note that we do not assume the functions in $S(B,D,C)$ share a common domain; the set $\Omega$ is not fixed.

\vspace{0.3cm}

\begin{thm}\label{pmean}
In the setting described above, there exists a constant $\tilde{C}_p$ such that $f\in S(B,D,C)\Rightarrow f^p\in S(B,D,\tilde{C}_p)$ where $0<p<1$. That is, for a choice of ``ball" and a dilation structure on $\mathbb{R}^n$, provided a suitable radius function exists, we have that whenever $f:\Omega\rightarrow\mathbb{R}$ is a locally bounded non-negative function on an open set $\Omega$ having
$$f(a)\leq \frac{C}{r^A}\int_{B_r(a)}f(x)\,dx$$
whenever $\overline{B_r(a)}\subseteq\Omega$, then we also have
$$f(a)^p\leq \frac{\tilde{C}_p}{r^A}\int_{B_r(a)}f(x)^p\,dx$$
whenever $\overline{B_r(a)}\subseteq\Omega$.

Moreover, we can take $\tilde{C}_p$ to be $2R(0)^{-A}(2K^A)^{(1-p)/p}C$.
\end{thm}

\begin{proof}
Observe that for each $a\in\mathbb{R}^n$, $s>0$, and $f\in S(B,D,C)$, the function $\tilde{f}$ given by $\tilde{f}(\cdot)=f(a+\delta_s(\cdot))$ is also in $S(B,D,C)$. This is because the collection of mean value inequalities is invariant under the change of variables $y=a+\delta_s(\tilde{y})$, explicitly we have
\begin{align*}&f(y_0)\leq \frac{C}{r^A}\int_{B_r(y_0)}f(y)\,dy\\
\Leftrightarrow\quad &\tilde{f}(\tilde{y}_0)\leq C\left(\frac{s}{r}\right)^A\int_{B_{r/s}(\tilde{y}_0)}\tilde{f}(\tilde{y})\,d\tilde{y}
\end{align*}
where $\tilde{y}_0=(\delta_{1/s}(y_0-a))$. Hence for a given $r>0$, $a\in\mathbb{R}^n$, the inequality
\begin{equation}\label{ineq4}
f(a)^p\leq \frac{\tilde{C}}{r^A}\int_{B_r(a)}f(x)^p\,dx
\end{equation}
for $\overline{B_r(a)}\subseteq\Omega$, is equivalent by the change of variables $x=a+\delta_r(y)$ to
\begin{equation}\label{ineq5}
\tilde{f}(0)^p\leq \tilde{C}_p\int_B\tilde{f}(y)^p\,dy
\end{equation}
for the function $\tilde{f}\in S(B,D,C)$ given by $\tilde{f}(\cdot)=f(a+\delta_r(\cdot))$. Thus if we prove the inequality (\ref{ineq5}) for each $\tilde{f}\in S(B,D,C)$ defined in a neighbourhood of $\overline{B}$, the inequality (\ref{ineq4}) holds for each $r>0$, $a\in\mathbb{R}^n$.

Relabelling $\tilde{f}$ by $f$, it remains to prove inequality (\ref{ineq5}) for each $f\in S(B,D,C)$ defined in a neighbourhood of $\overline{B}$. It is assumed that $f$ is locally bounded, thus bounded on $B$. Now note that the radius function $R$ is necessarily bounded above by $1$, for if there were $x$ such that $R(x)>1$, then the volume $|B_{R(x)}(x)|=R(x)^A|B|>|B|$ but $B_{R(x)}(x)\subseteq B$, a contradiction. Hence $f(x)^pR(x)^A$ is bounded on $B$. Thus there exists $a\in B$ such that
\begin{equation}\label{rwb}
f(x)^pR(x)^A\leq 2f(a)^pR(a)^A.
\end{equation}
The result is trivial if $f(a)=0$ as then $f$ is identically $0$ on $B$, so suppose $f(a)\neq 0$. Now in particular, for $x\in B_{R(a)}(a)$ we have
\begin{equation}\label{rwb2}
f(x)^p\leq 2f(a)^p\left(\frac{R(a)}{R(x)}\right)^A\leq 2K^Af(a)^p.
\end{equation}
We may raise this to the power $(1-p)/p$ to obtain
\begin{equation}\label{rwb3}
f(x)^{1-p}\leq (2K^A)^{(1-p)/p}f(a)^{1-p}.
\end{equation}
Now apply the mean value inequality on $B_{R(a)}(a)$. We have
$$f(a)\leq \frac{C}{R(a)^A}\int_{B_{R(a)}(a)}f(x)\,dx.$$
Multiplying both sides by $R(a)^A$ and using the bound (\ref{rwb3}) we obtain
$$f(a)R(a)^A\leq (2K^A)^{(1-p)/p}Cf(a)^{1-p}\int_{B_{R(a)}(a)}f(x)^p\,dx.$$
Cancelling $f(a)^{1-p}$ from both sides gives
$$f(a)^pR(a)^A\leq (2K^A)^{(1-p)/p}C\int_{B_{R(a)}(a)}f(x)^p\,dx.$$
We can lower bound the expression on the left by $(1/2)R(0)^Af(0)^p$ by using the bound (\ref{rwb}), and we can bound the integral on the right from above by the integral of $f^p$ over all of $B$. Rearranging gives the desired inequality with $\tilde{C}_p=2R(0)^{-A}(2K^A)^{(1-p)/p}C$.
\end{proof}

In fact, there is a further generalisation that we could make, observed by Riihentaus \cite{riihentaus2001generalized}, that the same proof works for concave surjections satisfying the ``$\Delta_2$-condition". This generalisation was established for averages over Euclidean balls, but fits with our generalisation. We have:

\vspace{0.3cm}

\begin{thm}
Let $(B,D)$ be a system of balls and suppose there exists a radius function $R$ as in the previous theorem. Now let $\phi: \mathbb{R}_{\geq 0}\rightarrow\mathbb{R}_{\geq 0}$ be a concave surjection with inverse satisfying the $\Delta_2$-condition, that is, there is a constant $c_\phi\geq 1$ so that for every $t$ we have $\phi^{-1}(2t)\leq c_\phi\phi^{-1}(t)$. Then there exists $C_\phi$ depending on $C$, $K$, $A$, $R(0)$ and $c_\phi$ so that every $f\in S(B,D,C)$ has $\phi\circ f\in S(B,D,C_\phi)$.
\end{thm}

We highlight the modifications of Riihentaus needed for the proof of Theorem \ref{pmean} to generalise. Equation (\ref{rwb}) should be replaced by the existence of an $a\in B$ so that $\phi(f(x))R(x)^A\leq 2\phi(f(a))R(a)^A$. Then equation (\ref{rwb2}) holds with $\phi\circ f$ in place of $f^p$. In particular we have $\phi(f(x))\leq 2^m\phi(f(a))$ in $B_{R(a)}(a)$ for some $m\in\mathbb{N}$ depending on $K$ and $A$. Iterating the $\Delta_2$-condition gives $f(x)\leq c_\phi^mf(a)$ in $B_{R(a)}(a)$.

We note that as $\phi$ is concave and $\phi(0)=0$, the function $t/\phi(t)$ is non-decreasing. Thus
$$\frac{u(y)}{\phi(u(y))}\leq\frac{c_\phi^mu(a)}{\phi(c_\phi^mu(a))}\leq c_\phi^m\frac{u(a)}{\phi(u(a))}$$
as $c_\phi\geq 1$. The rest of the proof is as before - use the mean value inequality on $B_{R(a)}(a)$, multiply and divide by $\phi(f(x))$ in the integral, use the bound just derived and $\phi(f(x))R(x)^A\leq 2\phi(f(a))R(a)^A$ to conclude.

Next, we shall establish the earlier claim that radius functions exist in the setting described above. We have:

\vspace{0.3cm}

\begin{prop}\label{radfun}
Let $B$ be an open, bounded subset of $\mathbb{R}^n$ containing $0$, $D$ a group of dilations $\delta_r$ given by $\delta_r(x_1,\dots,x_n)=(r^{\lambda_1}x_1,\dots,r^{\lambda_n}x_n)$ for some positive numbers $\lambda_1,\dots,\lambda_n$, and $B_r(a)$ be the associated balls. Then there exists a positive lower semi-continuous function $R$ on $B$ such that for all $a\in B$, $B_{R(a)}(a)\subseteq B$, and $R(a)/R(x)\leq 4$ whenever $x\in B_{R(a)}(a)$. In particular, $R$ is measurable.
\end{prop}
We remark that in this proof, the radius function we obtain is probably quite bad in many cases, in the sense that the value of $R(0)$ is probably small (and hence $\tilde{C}_p$ is large) unless $B$ is quite nice, for example, if $B$ is convex. This is because we essentially construct a radius function adapted to boxes, rather than the geometry of the set, which we use to regain the property that $r<s \Rightarrow B_r\subseteq B_s$, which can fail in general. Also, if all $\lambda_i\geq 1$, we can replace the definition of $R$ in the proof with its double, and obtain further that $R(a)/R(x)\leq 2$ for $x\in B_{R(a)}(a)$ by the obvious modifications.
\begin{proof}
Pick a box $\tilde{B}=(-y_1,y_1)\times\dots\times (-y_n,y_n)$ containing $B$, and use $\tilde{B}_r(a)$ to denote the system of balls associated to the pair $(\tilde{B},D)$. For $a\in B$, let $R(a):=\sup\{r>0: \tilde{B}(a)\subseteq B\}/4$. We will see that this $R$ has the desired properties.

Firstly, note that as $B$ is open and the diameter of $\tilde{B}_r$ tends to $0$ as $r\rightarrow 0$, the set in the definition of $R$ is non-empty so $R$ is well-defined and positive. Furthermore, $\tilde{B}_r\subseteq\tilde{B}_s$ whenever $r<s$, and by definition there is $s>R(a)$ with $\tilde{B}_s\subseteq B$, hence $B_{R(a)}(a)\subseteq \tilde{B}_{R(a)}(a)\subseteq \tilde{B}_s(a)\subseteq B$.

Also, $R$ is lower semi-continuous. We need to check that for any $y<R(a)$, there is a neighbourhood of $a$ so that $R(x)>y$ in that neighbourhood. By the definition of $R$, there is $r>y$ so that $\tilde{B}_{4r}(a)\subseteq B$. Choose $r_0\in (y,r)$. It suffices to find a neighbourhood $U$ of $a$ so that $\tilde{B}_{4r_0}(x)\subseteq B$ for $x$ in that neighbourhood. Take $U=\{x:|x_i-a_i|<[(4r)^{\lambda_i}-(4r_0)^{\lambda_i}]\cdot|y_i|\}$. It immediately follows from the triangle inequality that $\tilde{B}_{4r_0}(x)\subseteq \tilde{B}_{4r}(a)$. 

Finally, let $x\in B_{R(a)}(a)$. We will show $R(a)/R(x)\leq 4$. Note that also $x\in\tilde{B}_{R(a)}(a)$, and as we are considering dilates of an open box, we have in particular that $x\in\tilde{B}_r(a)$ for some $r<R(a)$. We claim that $\tilde{B}_{R(a)}(x)\subseteq B$.

As $x\in\tilde{B}_r(a)=a+\tilde{B}_r$, we have $\tilde{B}_{R(a)}(x)=x+\tilde{B}_{R(a)}\subseteq a+\tilde{B}_r+\tilde{B}_{R(a)}$. We have
$$\tilde{B}_r+\tilde{B}_{R(a)}=\prod_{i=1}^n(-(r^{\lambda_i}+R(a)^{\lambda_i})y_i,(r^{\lambda_i}+R(a)^{\lambda_i})y_i).$$
But $(-(r^{\lambda_i}+R(a)^{\lambda_i})y_i,(r^{\lambda_i}+R(a)^{\lambda_i})y_i)\subseteq (-2(r+R(a))^{\lambda_i}y_i,2(r+R(a))^{\lambda_i}y_i)$ since $z_1^{\lambda_i}+z_2^{\lambda_i}\leq 2(z_1+z_2)^{\lambda_i}$\,\footnote{$z_1^{\lambda_i}+z_2^{\lambda_i}\leq2\max(z_1,z_2)^{\lambda_i}\leq2(z_1+z_2)^{\lambda_i}$, or if $\lambda_i\geq 1$ we get $z_1^{\lambda_i}+z_2^{\lambda_i}\leq(z_1+z_2)^{\lambda_i}$ by convexity.}. Hence $\tilde{B}_r+\tilde{B}_{R(a)}\subseteq \tilde{B}_{2(r+R(a))}$ and thus $\tilde{B}_{R(a)}(x)\subseteq \tilde{B}_{2(r+R(a))}(a)\subseteq B$, the last inclusion by the definition of $R$, that $2(r+R(a))<4R(a)$ and the fact that the dilates of $\tilde{B}$ are nested.

It follows that $R(x)\geq (1/4)R(a)$, which completes the proof.
\end{proof}

Finally, we make some comments about heatballs, and more generally mean value inequalities where the ``centre" is not in the interior of the object. Specifically, let $B$ be some open bounded set not containing $0$, and consider the system of balls generated by dilating then translating this set. Do mean value inequalities for non-negative functions $f$ in this system imply mean value inequalities for $f^p$?

Our proof does not generalise in a straightforward way to this case. However, we can consider another open, bounded set $\tilde{B}$ containing $B\cup\{0\}$, and take that to be our new unit ball. Since $B_r(a)\subseteq \tilde{B}_r(a)$ for all $a$ and $r$, and we are considering non-negative functions in particular, it is clear that integrals over $B_r(a)$ are bounded above by integrals over $\tilde{B}_r(a)$ and hence mean value inequalities associated to $(B,D)$ imply mean value inequalities associated to $(\tilde{B},D)$, that is, $S(B,D,C)\subseteq S(\tilde{B},D,C)$. Theorem \ref{pmean} applies and gives $f\in S(B,D,C)\Rightarrow f^p\in S(\tilde{B},D,\tilde{C}_p)$.

This is sufficient for our purposes, but begs the question of whether we can avoid replacing $B$ by $\tilde{B}$, perhaps at least in the case where $0$ is on the boundary of $B$, as in the heatball case.

\subsection{Completing the proofs}\label{complete}
\begin{proof}[Proof of Theorem \ref{LapThm}]
We first pick $R_1>0$ so that $\Omega_{R_1}$, the set of points $x$ in $\Omega$ such that $\overline{B_{R_1}(x)}\subseteq\Omega$, has positive measure. Note that $\Omega_{R_1}$ is open and bounded. Choose $R_2>0$ such that $(\Omega_{R_1})_{R_2}=\Omega_{R_1+R_2}$ has positive measure. Now let $c>0$ be such that $c-K_nR_2^2<0$, where $K_n=1/(2n+4)$ is as in Claim \ref{LapClaim}. Then by Claim \ref{LapClaim}, if $u\leq c$ on $\Omega_{R_1}$, we have $u\leq c-K_nR_2^2$ on $\Omega_{R_1+R_2}$ and hence $\|u\|_{L^p(\Omega)}\geq\|u\|_{L^p(\Omega_{R_1+R_2})}\geq |c-K_nR_2^2|\cdot|\Omega_{R_1+R_2}|^{1/p}$. Otherwise there exists $x_0\in\Omega_{R_1}$ such that $u(x_0)>c$.

By Corollary \ref{LapCor} we have that the positive part of $u$, $u_+(x)=\max(u(x),0)$, satisfies a mean value inequality. We apply Theorem \ref{pmean}. In the terminology of the theorem, we have $C=1/|B_1|$, and radius function $R(a)=(1-|a|)/2$, $K=2$ and $A=n$. Thus for each $p\in (0,1)$ we get a mean value inequality for $u_+^p$, with constant $\tilde{C}_p=2R(0)^{-A}(2K^A)^{(1-p)/p}C=2^{1-n+\frac{(n+1)(1-p)}{p}}/|B_1|$, in particular, applied at $x_0$ we obtain
$$c^p\leq u_+^p(x_0)\leq\frac{\tilde{C}_p}{R_1^n}\int_{B_{R_1}(x)}u_+^p(x)\,dx.$$
Bounding the right hand integral by the integral of $|u|^p$ over all of $\Omega$, rearranging and taking $p^{th}$ roots we obtain
$$\|u\|_{L^p(\Omega)}\geq c\left(\frac{R_1^n}{\tilde{C}_p}\right)^{1/p}=c\left(\frac{|B_{R_1}|}{2^{(n+1-2np)/p}}\right)^{1/p}.$$
So in either case we have
$$\|u\|_{L^p(\Omega)}\geq\min\left(c\left(\frac{|B_{R_1}|}{2^{(n+1-2np)/p}}\right)^{1/p},|c-K_nR_2^2|\cdot|\Omega_{R_1+R_2}|^{1/p}\right).$$
\end{proof}

\begin{proof}[Proof of Theorem \ref{HeatThm}]
Fix a natural number $m\geq 3$ and consider an open box $\tilde{B}$ containing $E_m(0,0;1)$. For instance, an elementary examination of the expressions for heatballs given in section \ref{heatstart} shows that we can take $\tilde{B}=B_{m,n}\times(-(2\pi)^{-1},(2\pi)^{-1})$, where $B_{m,n}$ denotes the product of $n$ copies of the interval $(-(m+n)/(\pi e),(m+n)/(\pi e))$. As in section \ref{mvi}, we generate a family of balls $\tilde{B}_r(x,t)=(x,t)+\delta_r(\tilde{B})$, where $\delta_r(x,t)=(rx,r^2t)$ is parabolic scaling. It is clear that $\tilde{B}_r(x,t)$ contains $E_m(x,t;r)$.

For an open, bounded set $O$, we will denote by $O_R^{(m)}$ the set of points in $O$ such that $\overline{\tilde{B}_R(x,t)}$ is contained in $O$.

We first pick $R_1>0$ so that $\Omega_{R_1}$, the set of points $(x,t)$ in $\Omega$ such that $E(x,t;R_1)\subseteq\Omega$, has positive measure. Note that $\Omega_{R_1}$ is open and bounded.

Now we choose $R_2>0$ such that $\tilde{\Omega}:=(\Omega_{R_1})_{R_2}^{(m)}$ has positive measure. Let $c>0$ be such that $c-K_nR_2^2<0$, where $K_n=(n^2/2)|E(0,0;1)|e^{n+2}$ is as in Claim \ref{HeatClaim}. By Claim \ref{HeatClaim}, if $u\leq c$ on $\Omega_{R_1}$, we have $u\leq c-K_nR_2^2$ on $\tilde{\Omega}$ and hence $\|u\|_{L^p(\Omega)}\geq\|u\|_{L^p(\tilde{\Omega})}\geq |c-K_nR_2^2|\cdot|\tilde{\Omega}|^{1/p}$. Otherwise there exists $(x_0,t_0)\in\Omega_{R_1}$ such that $u(x_0,t_0)>c$.

By Proposition \ref{HeatMod} we have that $u_+$, the positive part of $u$, satisfies a mean value inequality. By the non-negativity of $u_+$, we may extend this to the mean value inequalities
$$u_+(x,t)\leq\frac{M_{m,n}}{r^{n+2}}\int_{\tilde{B}_r(x,t)}u_+(y,s)\,dy\,ds.$$
We apply Theorem \ref{pmean}. In the terminology of the theorem, we have $C=M_{m,n}$, $A=n+2$. The radius function is constructed by Proposition \ref{radfun}\,\footnote{By the remarks following that proposition, we can take $K=2$ in Theorem \ref{pmean} by choosing $R(x,t):=\sup\{r>0: \tilde{B}(x,t)\subseteq \tilde{B}\}/2$ in the proof, and clearly $R(0,0)=1/2$.}. Thus for each $p\in (0,1)$ we get a mean value inequality for $u_+^p$, with constant $\tilde{C}_p=2R(0,0)^{-A}(2K^A)^{(1-p)/p}C$, in particular, applied at $(x_0,t_0)$ we obtain
$$c^p\leq u_+^p(x_0,t_0)\leq\frac{\tilde{C}_p}{R_1^{n+2}}\int_{\tilde{B}_{R_1}(x)}u_+^p(x)\,dx.$$
Bounding the right hand integral by the integral of $|u|^p$ over all of $\Omega$, rearranging and taking $p^{th}$ roots we obtain
$$\|u\|_{L^p(\Omega)}\geq c\left(\frac{R_1^{n+2}}{\tilde{C}_p}\right)^{1/p}.$$
So in either case we have
$$\|u\|_{L^p(\Omega)}\geq\min\left(c\left(\frac{R_1^{n+2}}{\tilde{C}_p}\right)^{1/p},|c-K_nR_2^2|\cdot|\tilde{\Omega}|^{1/p}\right).$$

For the particular choices mentioned here, $\tilde{C}_p=2^{-n-1}(2^{n+3})^{(1-p)/p}M_{m,n}$, where $M_{m,n}$ is
$$M_{m,n}=|B_1|\frac{2\pi}{e}\left(\frac{2(m+n)(m+2)}{(4\pi e)(m-2)}\right)^{m/2}\left(\frac{m(m+n)}{m-2}\right)$$
and $|B_1|$ denotes the volume of a unit Euclidean ball in $\mathbb{R}^n$.
\end{proof}

\section{Further comments and results}\label{disc}
We collect here a few results of relevance to the main theorems and regarding the uniform sublevel set problem more generally.

\subsection{The failure of uniform sublevel set estimates for the heat operator in 1 spatial dimension}
With regards to the failure of uniform sublevel set estimates, we note that the counterexample for the Laplacian by Carbery-Christ-Wright \cite{carbery1999multidimensional} relies on Mergelyan's Theorem. One can extend this construction to other differential operators without any difficulty provided an analogue of Mergelyan's Theorem holds for that operator. This is a huge request, and although it may be possible, it is much more reasonable to work with Runge's Theorem, for which many generalisations have been considered - in particular, for the heat operator.

Both Runge's Theorem and Mergelyan's Theorem can be found in Rudin's book \cite{rudin1966real}, and the analogue of Runge's Theorem for the heat operator was given by Jones \cite{jones1975approximation}. Runge's Theorem and Mergelyan's Theorem concern holomorphic functions in the plane, but by considering the real and imaginary parts separately can be considered as a result concerning harmonic functions in the plane. We shall state the necessary consequences of these results during the proof of the forthcoming proposition.

We shall establish, as claimed in the introduction, that even with one spatial dimension we have no uniform sublevel estimate for the heat operator. At the same time, we shall re-establish the Carbery-Christ-Wright counterexample, using an alternative proof communicated by James Wright. We stress that this statement is by no means as general as we could make it, in light of other situations where a Runge-type theorem holds - see, for instance, Kalmes \cite{kalmes2019power}.

\vspace{0.3cm}

\begin{prop}\label{lhfailure}
Consider the operators $\Delta=\partial_{xx}^2+\partial_{tt}^2$ and $H=\partial_{xx}^2-\partial_t$ on $[0,1]^2$. For each $\varepsilon>0$, there exists a smooth $u$ with $\Delta u \equiv 1$ on $[0,1]^2$ but also satisfying the estimate $|\{x\in[0,1]^2:|u(x)|\leq\varepsilon\}|\geq 1-\varepsilon$. Furthermore, for each $\varepsilon>0$, there exists a smooth $u$ with $Hu \equiv 1$ on $[0,1]^2$ but also $|\{x\in[0,1]^2:|u(x)|\leq\varepsilon\}|\geq 1-\varepsilon$.
\end{prop}
\begin{proof}
We will prove both statements simultaneously, indicating the differences as they appear.

The Runge theorem for the Laplacian says that for a harmonic function on an open set $U$ containing a compact set $K$ to be uniformly approximated on $K$ by polynomials harmonic on all of $\mathbb{R}^2$, it is enough that the complement of $K$ is connected.

The Runge theorem for the heat operator says that a temperature (a solution to the heat equation) on an open set $U$ containing a compact set $K$ may be uniformly approximated on $K$ by temperatures on all of $\mathbb{R}^{n+1}$ provided that the $t$-slices of the complement of $U$ have no compact component, that is, the sets $\{x\in\mathbb{R}^n:(x,t)\in U^{c}\}$ have no compact component. For us, then, it is clearly sufficient that the $t$-slices of $U$ be intervals.

In each case, our compact set will be a collection of about $1/\delta$ rectangles of width $\delta$ separated by tiny gaps and covering most of the unit square. To be precise, consider some $0<\delta<1/2$ and let $K$ be the union of the disjoint rectangles
$$K:=\bigcup_{i=1}^{\left\lfloor\frac{4-\delta^2}{4\delta+\delta^2}\right\rfloor} \left[\frac{\delta}{4},1-\frac{\delta}{4}\right]\times\left[i\left(\delta+\frac{\delta^2}{4}\right)-\delta,i\left(\delta+\frac{\delta^2}{4}\right)\right],$$
see Figure \ref{Kset}. One sees that $\lfloor(4-\delta^2)/(4\delta+\delta^2)\rfloor(\delta+\delta^2/4)\leq 1-(\delta/4)$, and that the rectangles are separated by $\delta^2/4$, so for instance the $\delta^2/16$ neighbourhood of $K$ (say in the supremum norm) is also a disjoint collection of rectangles. This neighbourhood will be our open set $U$.

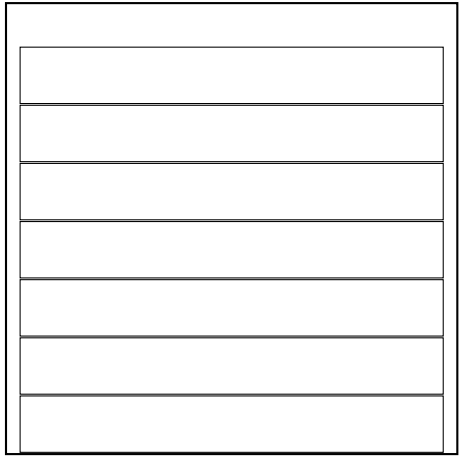
\begin{figure}
\begin{center}
\begin{tikzpicture}[x=6cm,y=6cm]
\draw[thick] (0,0) rectangle (1,1);
\draw (0.03125,0.00390625) rectangle (0.96875,0.12890625);
\draw (0.03125,0.1328125) rectangle (0.96875,0.2578125);
\draw (0.03125,0.26171875) rectangle (0.96875,0.38671875);
\draw (0.03125,0.390625) rectangle (0.96875,0.515625);
\draw (0.03125,0.51953125) rectangle (0.96875,0.64453125);
\draw (0.03125,0.6484375) rectangle (0.96875,0.7734375);
\draw (0.03125,0.77734375) rectangle (0.96875,0.90234375);
\end{tikzpicture}
\end{center}
\caption{The set $K\subseteq [0,1]^2$ for $\delta=1/8$.}\label{Kset}
\end{figure}

Consider $v(x,t)=t^2/2$ for the $\Delta$ case, $v(x,t)=-t$ for the $H$ case. Then $\Delta v$, respectively $Hv$, is identically $1$. Now on each of the thin rectangles of $U$, pick the $t$ coordinate of some point in the rectangle - call it $c$ - and define $w_1(x,t)$ on that component to be $c^2/2$ (respectively $-c$). Since $w_1$ is locally constant, it is harmonic (respectively, a temperature) on $U$. It is easily seen that, since the side length along the $t$ axis of each such rectangle is $\delta+(\delta^2/8)$, $w_1$ uniformly approximates $v$ on $U$ to within $\delta+(\delta^2/8)$.

By Runge's theorem for harmonic functions, it is clear that we can approximate $w_1$ uniformly to within $\delta$ on $K$ by a function $w_2$ harmonic on all of $\mathbb{R}^2$. Thus $w_2$ approximates $v$ on $K$ to within $2\delta+(\delta^2/8)$. Hence $u:=v-w_2$ satisfies $\Delta u \equiv 1$ and $|u(x)|\leq 2\delta+(\delta^2/8)$ on $K$. The analogous statement for the $H$ case is true, since the $t$-slices of $U$ are intervals, and so the we can apply the Runge theorem for temperatures. 

It remains to note that $K$ has large measure. Indeed, the measure of $K$ is
\begin{align*}
\delta\left(1-\frac{\delta}{2}\right)\left\lfloor\frac{4-\delta^2}{4\delta+\delta^2}\right\rfloor&\geq\delta\left(1-\frac{\delta}{2}\right)\left(\frac{4-\delta^2}{4\delta+\delta^2}-1\right)\\
&=\frac{\delta^3+3\delta^2-14\delta+8}{8+2\delta}>\frac{8-14\delta}{8+2\delta}=1-\frac{16\delta}{8+2\delta}\\
&>1-2\delta.
\end{align*}
Taking $2\delta+(\delta^2/8)\leq\varepsilon$ completes the proof.
\end{proof}

\subsection{Failure of many types of inequality for non-linear operators}
We are considering a hierarchy of inequalities associated to differential operators. Near the bottom of this hierarchy, we have lower bounds on the $L^1$ norm, which hold for any linear differential operator with sufficiently regular coefficients. Weaker than this was an estimate of the form
$$\|u\|_{L^p(\Omega)}\cdot|\{x\in \Omega:|u(x)|\geq c\}|^{1/p'}\geq c$$
for some $1\leq p\leq\infty$. In general, the main structural aspect of this inequality is that we are using some $L^p$ norm to balance the potentially small contribution of the superlevel set term raised to some power.

It makes sense to ask whether such inequalities hold in the non-linear case, or more generally for classes of real-valued functions not arising from having some differential operator applied to it be bounded away from $0$, or perhaps for complex-valued functions having some derivative bounded away from $0$ in modulus. This is perhaps unlikely to be of much interest in general; nevertheless, for a fairly simple non-linear differential operator, we can obtain a rather striking failure of such inequalities - we can find a sequence of functions with $Du_N\geq 1$ but $u_N$ uniformly tending to $0$ - so in fact, the superlevel set is empty for some functions, and we cannot obtain a lower bound even on the $L^\infty$ norm.

We shall consider $Du=-\det \text{Hess }u=(\partial^2_{xy}u)^2-(\partial^2_{xx}u)(\partial^2_{yy}u)$ on $[0,1]^2$. In general, the determinant of the Hessian of a function on $\mathbb{R}^n$ is a natural quantity to consider, being the product of the eigenvalues of the Hessian, which we can think of as being a measure of curvature in some sense (indeed, at critical points, this gives the product of the principal curvatures of the graph - the Gaussian curvature). It has some homogeneity $D(\lambda u)=\lambda^nDu$ which means that sublevel set and $L^p$ lower bounds would scale.

Furthermore, under additional assumptions that $u$ is strictly convex and non-negative on a convex domain, there are uniform sublevel set estimates associated to the determinant of the Hessian with power $n/2$, see Carbery \cite{carbery2010uniform}. It is shown that uniform sublevel set estimates fail otherwise, but we shall consider an example from a paper of Gressman \cite{gressman2011uniform}, which considers some uniform sublevel set estimates and gives remarks on situations when the uniformity fails.

The family of functions we consider is $u_N(x,y)=N^{-1}(e^x\sin(Ny)+e)$. Note that these are non-negative on the convex domain $[0,1]^2$. Clearly we have that $Du_N=e^{2x}\geq 1$, but given any $c$, we can always take $N$ large enough that $\{x\in[0,1]^2:|u_N(x)|\geq c\}$ is empty, so none of the types of inequality we have been considering can possibly hold.

\subsection{Extensions of the main theorems}\label{exten}
It follows from the uniformity of the constant in the main theorems that we can extend them to some other situations. For instance, though we were implicitly considering highly regular functions in the proofs, uniformity allows us to take limits in the inequality to obtain results for more ``rough" functions. By smoothing out a function that only satisfies a differential inequality in a weak/distributional sense, we can apply the inequality first to smooth approximations and then use a standard limiting argument to conclude that it holds for more general functions. We shall give no precise formulations, but simply note that our results can thus be extended to greater generality if desired.

Another property is that we can lift these inequalities to higher dimensions, for instance, Theorem \ref{LapThm} also implies that a lower bound in $L^p$ holds for the Laplacian in the first two variables considered as a differential operator on $\mathbb{R}^3$. Concretely, we can say the following.

\vspace{0.3cm}

\begin{prop}\label{dimup}
Let $\Omega_1\subseteq\mathbb{R}^{n_1}$ be open and bounded and suppose that within some set of functions $S$ on $\Omega_1$, we have the lower bound $$\|u\|_{L^p(\Omega_1)}\geq c.$$
Let $\Omega_2$ be a bounded open set in $\mathbb{R}^{n_2}$ and suppose $\Omega\subseteq\mathbb{R}^{n_1+n_2}$ contains $\Omega_1\times\Omega_2$. Then for functions $v$ on $\Omega$ such that their restrictions to $\Omega_1\times\{y\}$ lies in $S$ for each $y\in\Omega_2$, we have
$$\|v\|_{L^p(\Omega)}\geq c|\Omega_2|^{1/p}.$$
\end{prop}
\begin{proof}
We write $(x,y)$ for an element of $\mathbb{R}^{n_1}\times\mathbb{R}^{n_2}$. Let $v$ be as in the statement, then $\|v(\cdot,y)\|_{L^p(\Omega_1)}\geq c$. It is clear that
$$\|v\|_{L^p(\Omega_1\times\Omega_2)}=\|(\|v(\cdot,y)\|_{L^p(\Omega_1)})\|_{L^p(\Omega_2)}\geq\|c\|_{L^p(\Omega_2)}=c|\Omega_2|^{1/p}.$$
It follows that $\|v\|_{L^p(\Omega)}\geq c|\Omega_2|^{1/p}$.
\end{proof}

We can also consider the effect of diffeomorphisms on these inequalities. Suppose we have an inequality $\|u\|_{L^p(\Omega)}\geq c$ for some collection of functions $u$ on $\Omega$ and a diffeomorphism $\phi:\Omega\rightarrow\Omega'$. By the change of variables formula we have
$$\int_\Omega|u(x)|^p\,dx=\int_{\Omega'}|u(\phi^{-1}(x'))|^p|\det J\phi^{-1}(x')|\,dx'$$
where $J\phi^{-1}$ is the Jacobian of $\phi^{-1}$. Let $M:=\sup\{|\det J\phi^{-1}(x')|:x'\in\Omega'\}$. Then we have
$$\|u\|_{L^p(\Omega)}\leq \|u\circ\phi^{-1}\|_{L^p(\Omega')}M^{1/p}.$$
Thus if $\|u\|_{L^p(\Omega)}\geq c$, we also have $\|u\circ\phi^{-1}\|_{L^p(\Omega')}\geq cM^{-1/p}$.
Hence we obtain lower bounds in $L^p$ for the collection of $u\circ\phi^{-1}$ on $\Omega'$. Note that passing between these classes of functions presents no loss when the Jacobian determinant is constant, such as in the case of invertible linear transformations.

An important consequence of this is the following: Suppose we know an $L^p$ lower bound holds in the class where $Du(x)\geq 1$, for $D$ a differential operator. For clarity say that $D$ is written in terms of ``$x$ coordinates". Then if we express $D$ in terms of $x'$ coordinates, with $\phi$ being the diffeomorphism giving this change of coordinates, we know that in the class where $D(u\circ\phi^{-1})(x')\geq 1$, an $L^p$ lower bound holds. In short, the class of differential operators for which $L^p$ lower bounds hold is invariant under change of coordinates.

As an example, let us consider which constant coefficient linear differential operators satisfy $L^p$ lower bounds. By using invertible linear transformations, one can reduce the cases to study to certain canonical forms. For example, the theory of quadratic forms tells us that to understand the homogeneous second order examples in $\mathbb{R}^n$, we need only understand those having associated polynomials of the form
$$\sum_{i=1}^{m_1}x_i^2-\sum_{j=m_1+1}^{m_2}x_j^2$$
for $m_2\leq n$.

In $\mathbb{R}^2$, this allows us to give a complete picture - in fact, $L^p$ lower bounds hold in all cases. The quadratic form $x_1^2+x_2^2$ corresponds to the Laplacian, hence follows from Theorem \ref{LapThm}. All the others satisfy the Carbery-Christ Wright Theorem, so in fact a uniform sublevel estimate holds. This is obvious in all cases but $x_1^2-x_2^2$, but this is $(x_1-x_2)(x_1+x_2)$, so setting $x=x_1-x_2$ and $y=x_1+x_2$, we obtain $xy$ via change of coordinates, which is of the correct form. Summarising, we have:

\vspace{0.3cm}

\begin{prop}
For every linear homogeneous second order differential operator $D$ in $\mathbb{R}^2$, $\Omega\subseteq\mathbb{R}^2$ open and bounded, $p\in(0,1)$, there exists a constant $c_p$ depending only on $\Omega$, $p$ and $D$ such that whenever $Du\geq 1$ in $\Omega$, we have $\|u\|_{L^p(\Omega)}\geq c_p$.
\end{prop}

\subsection{The $L^p$ means, $p\leq 0$}\label{Lp0}
In this section we will shall attempt to provide some evidence to suggest that uniform lower bounds on the $L^p$ quasi-norms are indeed a natural consideration in relation to the uniform sublevel set problem by observing that, for $p<0$, uniform lower bounds on the $L^p$ means are approximately equivalent to uniform sublevel set estimates. For the reader unfamiliar with these means, we shall review some basic facts, the proofs of these and more can be found in the chapter on integral means in the book of Hardy, Littlewood \& P\'olya \cite{hardy1952inequalities}.

Throughout this section, we will use $\|u\|_p$ to denote a normalised $L^p$ mean on a bounded open set $\Omega$. For $p\neq 0$, these are defined by
$$\|u\|_p:=\left(\frac{1}{|\Omega|}\int_\Omega |u(x)|^p\,dx\right)^{1/p}$$
and for $p=0$, we define
$$\|u\|_0:=\exp\left(\frac{1}{|\Omega|}\int_\Omega \log(|u(x)|)\,dx\right).$$
Here we are using the conventions that $0$ and $\infty$ are reciprocal to each other, and $\exp(-\infty)=0$. Note that the positive and negative parts of $\log(|u|)$ could have integral $\infty$ and $-\infty$ respectively, and the expression is not well defined. In this case it makes sense to set $\|u\|_0=\infty$. We also note that there is a ``$-\infty$" mean, the essential infimum of $|u|$, but this will not be important here.

Note that for $p\leq 0$, any function equal to $0$ on a set of positive measure will have $\|u\|_p=0$, hence non-zero functions can have $p$-mean equal to $0$, and so these are not quasi-norms - the word ``mean" is appropriate in this context.

With this normalisation, one has by H\"older's inequality that $\|u\|_p\leq\|u\|_q$ for $0<p\leq q$, and it turns out that this extends to the whole extended real line - we have $\|u\|_p\leq\|u\|_q$ whenever $p\leq q$. Extending the result to $0=p\leq q$ is a simple consequence of Jensen's inequality, since
\begin{align*}
\left(\exp\left(\frac{1}{|\Omega|}\int_\Omega \log(|u(x)|)\,dx\right)\right)^q&=\exp\left(\frac{1}{|\Omega|}\int_\Omega \log(|u(x)|^q)\,dx\right)\\
&\leq\frac{1}{|\Omega|}\int_\Omega |u(x)|^q\,dx.
\end{align*}
The result for $p\leq q\leq 0$ follows from the easy formula $\|u\|_p=(\|(1/u)\|_{-p})^{-1}$ which holds for all $p$ on the extended real line. The case $p\leq 0\leq q$ is immediate from these cases.

For reference, we note some common names for these means. The discrete analogues of the means for $p=1,0,-1$ are the arithmetic, geometric and harmonic means. The name ``geometric mean" is sometimes applied in the general case, as is the alternate name ``exp-log average". For $p\neq 0,\infty,-\infty$, these are referred to as power means, and for $p\neq\infty,-\infty$, they are examples of generalised means.

It is known that, provided that for some positive $p$, the $L^p$ mean of $u$ is finite, the limit as $p$ decreases to $0$ of the $p$-means is the $L^0$ mean. Thus it is natural to ask if, in a given set of functions, we have a lower bound on the normalised $L^p$ means for $p>0$ that is independent of $p$, for then we can pass to one for $L^0$ in the limit.

If possible, one will require some careful arguments. For instance, the reader will note that in Proposition \ref{dimup}, passing from a uniform lower bound on the $\Omega_1$ slices that is independent of $p$ to one on $\Omega$ is possible only for $\Omega_1\times\Omega_2$. Similarly, in the discussion following Proposition \ref{dimup}, one should ask that the Jacobian determinant be constant, as in the case of a linear transformation. Importantly, we see from applying the normalisation to the constants of section \ref{complete} that we cannot obtain a bound independent of $p$ in the main theorems - whether some finer methods can be used to demonstrate this remains to be determined.

Considering the $L^p$ means for negative $p$, we will show that (uniform) lower bounds on these $L^p$ means imply (uniform) sublevel set estimates, and we can also establish an approximate converse. We have by Chebyshev's inequality that
$$|\{x\in\Omega:|u(x)|\leq\varepsilon\}|\varepsilon^{p}\leq\int_\Omega |u(x)|^p\,dx$$
so upper bounds on the right hand side translate to sublevel set estimates with exponent $-p$ (note that $-p$ is positive). In particular, rearranging $\|u\|_p\geq c$ and applying the above gives
$$|\{x\in\Omega:|u(x)|\leq\varepsilon\}|\leq c^p|\Omega|\varepsilon^{-p}.$$
Certainly, then, we cannot have uniform lower bounds on $L^p$ means for any negative $p$ in the sets of functions considered in Theorems \ref{LapThm} and \ref{HeatThm}, for this would contradict the failure of uniform sublevel set estimates given by Proposition \ref{lhfailure} (recall that this family of counterexamples can be trivially extended to give counterexamples on cubes in higher dimensions).

This simple observation is not new, and there are already papers dealing with uniform upper bounds on integrals $\int_\Omega |u(x)|^{-\delta}\,dx$, $\delta>0$, see Phong-Stein-Sturm \cite{phong1999growth}, also Phong \& Sturm \cite{phong2000algebraic}. The only difference here is that we have rearranged this bound and stated it in the context of lower bounds on $L^p$ means, which perhaps suggests that this framework is a natural one. In this direction, we note also the paper of Nazarov-Sodin-Volberg \cite{nazarov2002geometric} which provides some lower bounds on $L^p$ means with $p\leq 0$ for polynomials on convex sets in terms of some suitable quantities.

Further evidence to support this naturality comes in the form of an approximate converse to this implication - that if we can obtain a uniform sublevel set estimate
$$|\{x\in\Omega:|u(x)|\leq\varepsilon\}|\leq C\varepsilon^\delta$$
then we also have uniform lower bounds on the each $L^p$ mean for $p>-\delta$.

Let $p\in(-\delta,0)$ and let $k_0$ be the smallest integer $k$ such that $2^{k\delta}C\geq|\Omega|$. Then we may decompose $\Omega$ into sets $E_k:=\{x\in\Omega:2^{k-1}<|u(x)|\leq 2^k\}$ for $k\leq k_0$, along with $F:=\{x\in\Omega:2^{k_0}<|u(x)|\}$ and $G:=\{x\in\Omega:u(x)=0\}$. For $F$ we use the trivial estimate $|F|\leq|\Omega|$, and for the other two sets we have $|E_k|\leq 2^{k\delta}C$ and $|G|=0$ by the sublevel set estimate. Hence
\begin{align*}
\int_\Omega|u(x)|^p\,dx&=\sum\limits_{k\leq k_0}\int_{E_k}|u(x)|^p\,dx+\int_F|u(x)|^p\,dx+\int_G|u(x)|^p\,dx\\
&\leq\sum\limits_{k\leq k_0}2^{k\delta+(k-1)p}C+2^{k_0p}|\Omega|\\
&=2^{-p+k_0(\delta+p)}C\sum\limits_{k\leq 0}2^{k(\delta+p)}+2^{k_0p}|\Omega|\\
&=\frac{2^{-p+k_0(\delta+p)}C}{1-2^{-\delta-p}}+2^{k_0p}|\Omega|.
\end{align*}
Thus we obtain
$$\|u\|_p\geq\left(\frac{2^{-p+k_0(\delta+p)}C}{(1-2^{-\delta-p})|\Omega|}+2^{k_0p}\right)^{1/p}$$
where we note that the right hand side depends only on $C$, $\delta$, $\Omega$ and $p$, so uniformity of the constant in the sublevel set estimate gives uniformity in this lower bound.

One might ask if the converse actually holds, that is, does a positive lower bound on an $L^p$ mean hold if and only if sublevel set estimates hold with exponent $-p$? A negative answer is given by the standard examples $x^k$ on $(0,1)$, $k\in\mathbb{N}$. These satisfy sublevel set estimates with exponent $1/k$. However, their $L^{-1/k}$ means are $0$, since
$$\left(\int_0^1(x^k)^{-1/k}\,dx\right)^{-k}=\left(\int_0^1\frac{1}{x}\,dx\right)^{-k}=0.$$

\section{Some questions}\label{ques}
The central question arising from this paper concerns the generality in which $L^p$ lower bounds hold. We pose the following question.

\textbf{Question.} Given a linear differential operator $D$ on an open set $\Omega$, does there exist an exponent $p\in (0,1)$ and a constant $c_p>0$ depending on $D$, $n$, $\Omega$ and $p$ such that $\|u\|_{L^p(\Omega)}\geq c_p$ holds whenever $Du\geq 1$ in $\Omega$?

We could also pose this for other classes of functions, for non-linear differential operators, and possibly for $1\leq p\leq\infty$, but these are not the main cases of interest for us.

It could be that, just as for the $L^1$ norm, such inequalities hold for any linear differential operator. However, a reflection on the above proofs suggests that either another approach is needed to show this, or, if there are counterexamples, where we might start to look for them. One of the key elements of the proofs was the relation of a rate of change for a parameterised family of averages to the differential operator in question, either by means of a derivative or just a nicely-quantified difference, such that we can bound the difference from below in a uniform way.

Parameterised families of averages such as this occur in the study of mean value formulae for PDE, where often one seeks to establish that such a family of averages is constant if and only if a function satisfies a certain PDE or family of PDEs. In fact, if we allow for averages against measures that are not necessarily positive, a vast number of linear PDE solutions can be characterised this way. Pokrovskii \cite{pokrovskii1998mean} establishes a method for constructing such measures, generalising the ideas of Zalcman \cite{zalcman1973mean}.

The step involving passage from mean value inequalities for $u$ to mean value inequalities for the positive part $u_+$, and hence for $u_+^p$, seems to require that the measure be positive, which appears to be a great restriction. We also need our averages to be with respect to a measure which can at least be locally bounded by some power of the Lebesgue measure.

That such convenient families of averages exist is perhaps related to the existence of adequate theories of subsolutions and supersolutions where we have access to results such as maximum principles. Indeed, our arguments were based on the idea that for such operators, inequalities of the form $Du\geq 1$ represent a stronger property than that of a usual subsolution, where by comparison with a constant solution we would expect a quantifiably large deviation. This naturally suggests the possibility of extensions of the main results to more general elliptic and parabolic operators, and that counterexamples might be found by considering those differential operators for which such theories do not hold, such as the wave operator.

However, we must be careful to remember that this is not the only way that $L^p$ lower bounds can arise. Consider the wave operator $\partial^2_t-\Delta_x$. In one spatial dimension, the wave operator can be expressed as $\partial^2_t-\partial^2_x=(\partial_t-\partial_x)(\partial_t+\partial_x)$, which as we saw before can be written as $\partial_{x'}\partial_{y'}$ by the change of co-ordinates $x'=t-x, y'=t+x$, and so this in fact satisfies a uniform sublevel set estimate.

But in more than one spatial dimension, simply by considering functions that are constant in $t$, we see that no uniform sublevel set estimate holds due to the failure for the Laplacian. However, the wave operator is neither elliptic nor parabolic, which suggests that the methods of this paper may not be helpful for establishing an $L^p$ lower bound, so this could be an interesting case to consider.

Though tangential to the main question of the paper, it is also worth briefly posing some questions based on the discussion of mean value inequalities from section \ref{mvi}. Namely, in what generality can we prove results for passing from mean value inequalities for $f$ to those for $f^p$? For which families of ``balls" does this make sense? Could this be done in settings other than $\mathbb{R}^n$, such as Lie groups or more general manifolds? Additionally, for a given family of balls can we determine the best constant $\tilde{C}_p$ in Theorem \ref{pmean}?

Finally, we note that in the context of section \ref{Lp0}, it makes sense to ask more generally: for which sets of functions on an open, bounded $\Omega$ does there exist $p\in \mathbb{R}$ and a constant $c_p>0$ depending on $n$, $\Omega$ and $p$ such that the normalised mean $\|u\|_p\geq c_p$ whenever $u$ is in that set of functions? In particular, we ask this for sets of functions satisfying $Du\geq 1$ for some linear differential operator $D$, and furthermore it would be interesting to see if Theorems \ref{LapThm} and \ref{HeatThm} can be extended to $p=0$.

\textit{Acknowledgements.} The author is supported by a UK EPSRC scholarship at the Maxwell Institute Graduate School. The author would like to thank Prof. James Wright for many helpful discussions, along with various suggestions and improvements; in particular the inclusion of the discussion on failure of uniform sublevel set estimates via Runge-type theorems in section \ref{disc}. The author would also like to thank the referee of an earlier submission, from which many parts were substantially rewritten and included in the present paper.

\appendix
\section{Calculation of the maximum}\label{maxcalc}
In this section, we shall determine the maximum of the function
\begin{equation}\label{kmn}
\kappa_{m,n}(y,s)=\frac{|B_1|}{2m+4}\tilde{A}^m\left(\frac{m(m+n)}{s}\log\left(\frac{1}{4\pi s}\right)+\frac{|y|^2}{s^2}\right)
\end{equation}
over the set
$$E=\{(y,s)\in\mathbb{R}^{n+1}:0\leq s\leq 1/4\pi,|y|^2\leq 2s(m+n)\log(1/4\pi s)\}.$$
One can differentiate equation (\ref{kmn}) with respect to $y_i$ to obtain
\begin{align*}
\frac{\partial\kappa_{m,n}}{\partial y_i}(y,s)&=\left(\frac{m}{2}\tilde{A}^{m-2}(-2y_i)\right)\left(\frac{m(m+n)}{s}\log\left(\frac{1}{4\pi s}\right)+\frac{|y|^2}{s^2}\right)+\tilde{A}^m\frac{2y_i}{s^2}\\
&=\tilde{A}^{m-2}y_i\left(\frac{2\tilde{A}^2}{s^2}-m\left(\frac{m(m+n)}{s}\log\left(\frac{1}{4\pi s}\right)+\frac{|y|^2}{s^2}\right)\right).
\end{align*}
Clearly $\tilde{A}\neq 0$ in the interior of $E$, so at the maximum of $\kappa_{m,n}$ either $y_i=0$ or the term in brackets is $0$. Substituting $\tilde{A}$, we see that this is equal to
$$-\frac{(m+2)}{s^2}\left(|y|^2+(m-2)(m+n)s\log\left(\frac{1}{4\pi s}\right)\right).$$
However, as $m\geq 3$ and $0<s<(4\pi)^{-1}$, this quantity is always negative, so the maximum of $\kappa_{m,n}$ must occur with $y_i=0$. So it remains to see where $\kappa_{m,n}(0,s)$ is maximised. We have
$$\kappa_{m,n}(0,s)=\left[\frac{|B_1|(2(m+n))^{m/2}m(m+n)}{(2m+4)}\right]s^{(m-2)/2}\log\left(\frac{1}{4\pi s}\right)^{(m+2)/2}.$$
To find where the maximum occurs, we compute
\begin{align*}
&\frac{d}{ds}\left(s^{(m-2)/2}\log\left(\frac{1}{4\pi s}\right)^{(m+2)/2}\right)\\
=\,&s^{(m-4)/2}\left(\log\left(\frac{1}{4\pi s}\right)\right)^{m/2}\left(\frac{(m-2)}{2}\log\left(\frac{1}{4\pi s}\right)-\frac{(m+2)}{2}\right).
\end{align*}
As $0<s<(4\pi)^{-1}$, this can only be $0$ if the term in brackets is equal to $0$, which gives
$$\log\left(\frac{1}{4\pi s}\right)=\frac{m+2}{m-2},\quad s=(4\pi e^{(m+2)/(m-2)})^{-1}.$$
Thus the maximum $M_{m,n}$ of $\kappa_{m,n}$ is
$$\left[\frac{|B_1|(2(m+n))^{m/2}m(m+n)}{(2m+4)}\right](4\pi e^{(m+2)/(m-2)})^{-(m-2)/2}\left(\frac{m+2}{m-2}\right)^{(m+2)/2}.$$
We regroup the terms to a more convenient form:
\begin{align*}
M_{m,n}&=|B_1|\frac{2\pi}{e}\left(\frac{2(m+n)(m+2)}{(4\pi e)(m-2)}\right)^{m/2}\left(\frac{(m+2)m(m+n)}{(m-2)(m+2)}\right)\\
&=|B_1|\frac{2\pi}{e}\left(\frac{2(m+n)(m+2)}{(4\pi e)(m-2)}\right)^{m/2}\left(\frac{m(m+n)}{m-2}\right).
\end{align*}

\newpage
\bibliographystyle{abbrv}
\bibliography{Bibliography1}

John Green,\\ Maxwell Institute of Mathematical Sciences and the School of Mathematics,\\ University of Edinburgh,\\ JCMB, The King’s Buildings,\\ Peter Guthrie Tait Road,\\ Edinburgh, EH9 3FD,\\ Scotland\\ Email: \texttt{J.D.Green@sms.ed.ac.uk}
\end{document}